\documentclass[12pt]{amsart}
\usepackage{amssymb}
\usepackage{amsmath}
\def\D {{\mathcal D}}

\def \and{\qquad\text{and}\qquad}

\newcommand{\Id}{{1\hspace{-4pt} 1}}

\newtheorem{proposition}{Proposition}[section]
\newtheorem{theorem}[proposition]{Theorem}
\newtheorem{corollary}[proposition]{Corollary}
\newtheorem{lemma}[proposition]{Lemma}
\theoremstyle{definition}
\newtheorem{definition}[proposition]{Definition}
\theoremstyle{remark}
\newtheorem{remark}[proposition]{Remark}

\numberwithin{equation}{section}

\begin{document}

\title[On a rigidity property of perturbations of circle bundles on $3$-manifolds]{On a rigidity property of perturbations of circle bundles on $3$-manifolds}

\author{Massimo Villarini}
\footnote[1]{Massimo Villarini, Dipartimento di Scienze Fisiche, Informatiche e Matematiche, via Campi 213/b 41100, Universit\'a di Modena e Reggio Emilia, Modena, Italy

E-mail: massimo.villarini@unimore.it}

\begin{abstract}
Let $\mathcal{F}_0$ be a foliation  whose leaves are the fibers of a smooth circle bundle $\xi_0$ and let $P$, total space of $\xi_0$, be a closed oriented $3$-manifold. Let  $\varepsilon \rightarrow \mathcal{F}_{\varepsilon}$ be a smooth deformation of $\mathcal{F}_0$, each $\mathcal{F}_{\varepsilon}$ being a foliation by circles of $P$. We prove that if there exists a smooth curve $\varepsilon \rightarrow \gamma_{\varepsilon}$, $ \gamma_{\varepsilon}$ leaf of $\mathcal{F}_{\varepsilon}$, converging to a leaf $\gamma_0$ of $\mathcal{F}_0$ when $\varepsilon \rightarrow 0$, then the leaves of $\mathcal{F}_{\varepsilon}$ are the fibers of a circle bundle $\xi_{\varepsilon}$, and there exists a smooth $1$-parameter family of bundle isomorphisms $\varepsilon \rightarrow \varphi_{\varepsilon}$, such that $\varphi_0 = identity$ and $\varphi_{\varepsilon}^{-1}  (\xi_0) = \xi_{\varepsilon}$. In particular, this rigidity property of deformations of a circle bundle on $3$-manifolds always holds true for real analytic families $\varepsilon \rightarrow \mathcal{F}_{\varepsilon}$ if the base space of $\xi_0$ is not a torus. These results stem from the following principle: a foliation by circles $\mathcal{F}$ of $P$ which is sufficiently $C^1$-close and tangent to $\mathcal{F}_0$ along a fiber $\gamma_0$, {\it i.e.} $\mathcal{F}$ and $\mathcal{F}_0$ have $\gamma_0$ as a common leaf, is conjugated to ${\mathcal{F}_0}$ through a bundle isomorphism.
\end{abstract}

\maketitle

\section{Introduction}
Let
$$
S^1 = \{ z \in \mathbb{C} : \vert z \vert =1 \}
$$
and let
\begin{equation}\label{bundle}
\xi_0 \, : \, S^1 \hookrightarrow P \rightarrow^{\pi} M
\end{equation}
be a smooth (or real analytic) circle bundle. Throughout this article all mathematical objects will be smooth, {\it i.e.} infinitely differentiable, or real analytic: the results concerning the smooth case remain true if smoothness is understood as $C^k$-regularity, $k\geq 2$, for the hypotheses, and as $C^{k-1}$-regularity for the theses in the statements. Let
\begin{equation}\label{generatore}
X_0 : P \rightarrow TP
\end{equation}
be the fundamental vector field of $\xi_0$, {\it cfr.} \cite{kn} $\S 5$, defined by the differential of the $S^1$-action acting on the vector field $\frac{\partial}{\partial \theta}$ generating the Lie algebra of $S^1$, $z=e^{i \theta} \in S^1$. We will call $X_0$ the {\it isochronous infinitesimal generator of} $\xi_0$: its closed orbits are the fibers of $\xi_0$, and they all have minimal period $2 \pi$. $X_0$ completely describes $\mathcal{F}_0$, the foliation by circles whose leaves are the fibers of $\xi_0$, and our approach to the subject of perturbations of circle bundles on $3$-manifolds will be through the study of the class of vector fields satisfying the following definition

\begin{definition}\label{perturbazione}
Let $\overline \varepsilon >0$ and
$$
X: P \times ]-\overline \varepsilon , \overline \varepsilon[ \rightarrow T(P\times \mathbb{R})
$$
such that
\begin{equation}\label{campo}
\begin{split}
1) &X(\cdot , \varepsilon)  =  X_{\varepsilon}(\cdot) + \frac{\partial}{\partial \varepsilon}   \\
2) \mathcal{F}_{\varepsilon} & \, is \, a \, foliation \, by \, circles \, defined \, by \, X_{\varepsilon} : P \rightarrow TP .
\end{split}
\end{equation}
We will refer to $\varepsilon \rightarrow X_{\varepsilon}$ as a smooth deformation of $X_0$, and to $\varepsilon \rightarrow \mathcal{F}_{\varepsilon}$ as a smooth deformation of $\mathcal{F}_0$.
\end{definition}
The original motivations of this article were two questions, $(Q1)$ and $(Q2)$ below, concerning two basic results on foliations by circles, namely Seifert's Stability Theorem, and a result by D.B.A. Epstein characterizing foliations by circles on $3$-manifolds. Questions $(Q1)$, $(Q2)$ are related to a dynamical problem we will briefly mention at the end of this introduction.
\begin{theorem}\label{seifert}(Seifert's Stability Theorem \cite{seifert})
Let $\varepsilon \rightarrow X_{\varepsilon}$ be a smooth deformation of the isochronous infinitesimal generator $X_0$ of (\ref{bundle}) (here we do not assume that the integral curves of the $X_{\varepsilon}$'s, $\varepsilon \neq 0$, are closed, {i.e.} we do not assume $2)$ in Definition (\ref{perturbazione})). If $P$ is closed and the Euler characteristic $\chi (M) \neq 0$ there exists $\overline \varepsilon >0$ such that for every $\varepsilon$ satifying $\vert \varepsilon \vert < \overline \varepsilon$, there exists a closed orbit $\gamma_{\varepsilon}$ of $X_{\varepsilon}$ whose minimal period $T(\varepsilon)$ satisfies
$$
T(\varepsilon)=2 \pi + o(1)
$$
\end{theorem}
We will refer to the closed curves $\gamma_{\varepsilon}$ in the statement of the above theorem as {\it Seifert's leaves}

\begin{theorem}\label{epstein}(Epstein's Theorem \cite{epstein})
Any smooth foliation by circles generated by a $\mathbb{R}$-action on a closed $3$-manifold $P$ is diffeomorphically conjugated to a foliation generated by a $S^1$-action on $P$
\end{theorem}
It seems natural to ask:

$(Q1)$: when a family $\varepsilon \rightarrow \gamma_{\varepsilon}$ of Seifert's leaves is smooth (see definition below) and converges to a fiber $\gamma_0$ of $\xi_0$? Does the existence of this smooth family of Seifert's leaves simplify the dynamics of the perturbation in (\ref{campo}), which in general could be rather complicated, see {\it e.g.} \cite{sullivan}?

$(Q2)$: if in Epstein's Theorem we consider a $1$-parameter family of foliations by circles as in Definition (\ref{perturbazione}), do exsist diffeomorphisms between $\mathcal{F}_{\varepsilon}$ and $\mathcal{F}_0$ which form a smooth $1$-parameter family? are they bundles isomorphisms?

These two questions turn to be closely related and their investigation leads to the rigidity phenomenon referred to in the title of the article, {\it cfr.} Theorem (\ref{main}).

\begin{definition}\label{liscia}
Let $\varepsilon \rightarrow \mathcal{F}_{\varepsilon}$ be as in Definition (\ref{perturbazione}) and let $\varepsilon \rightarrow \gamma_{\varepsilon}$ be a $1$-parameter family of Seifert's leaves. Then $\varepsilon \rightarrow \gamma_{\varepsilon}$ is smooth and converges to $\gamma_0$, leaf of $\mathcal{F}_0$, if:
\begin{itemize}
\item [.] there exists a smooth curve $\varepsilon \rightarrow p_{\varepsilon}$ in $P$, $p_{\varepsilon} \in \gamma_{\varepsilon}$, such that $p_{\varepsilon} \rightarrow p_0$ as $ \varepsilon \rightarrow 0$, $p_0 \in \gamma_0$ \\
\item [.] denoting $(t, p, \varepsilon)  \rightarrow \phi^t_{\varepsilon} (p)$ the flows of $X_{\varepsilon}$, $\phi^t_{\varepsilon} (p_{\varepsilon}) \rightarrow \phi^t_0 (p_0)$ as $\varepsilon \rightarrow 0$, $t \in (-2 \pi , 2 \pi)$.
\end{itemize}
\end{definition}
\begin{remark} We explicitly observe that, from a straightforward application of the Implicit Function Theorem, if $\varepsilon \rightarrow \gamma_{\varepsilon}$ is a smooth curve of Seifert's leaves then as solutions of $X_{\varepsilon}$ each $\gamma_{\varepsilon}$ has minimal period $T(\varepsilon)= 2 \pi + o(1)$, and $\varepsilon \rightarrow T(\varepsilon)$  is smooth. Therefore:

if $\varepsilon \rightarrow \gamma_{\varepsilon}$ is a smooth curve of Seifert's leaves, up to smooth reparametrization of the vector fields $X_{\varepsilon}$, we can suppose that
$$
T(\varepsilon) \equiv 2 \pi.
$$
Basically we will always suppose this condition to hold in all the cases when smooth curves of Seifert's leaves appear in this article.
\end{remark}

The main result to be proved is
\begin{theorem}\label{main}
Let $\varepsilon \rightarrow \mathcal{F}_{\varepsilon}$ as in Definition (\ref{perturbazione}), and let $\varepsilon \rightarrow \gamma_{\varepsilon}$ be smooth in the sense of the previous definition. In (\ref{bundle}) let $dim P =3$, $P$, $M$ both closed and oriented. Then there exists a smooth $1$-parameter family of diffeomorphisms
$$
\varphi_{\varepsilon} : P \rightarrow P
$$
such that
\begin{itemize}
\item [(i)] $\varphi_0 = identity$ \\
\item [(ii)] the leaves of each $\mathcal{F}_{\varepsilon}$ are the fibers of a circle bundle 
$$
\xi_{\varepsilon} : S^1 \hookrightarrow P \rightarrow^{\pi} M
$$
and $\varphi_{\varepsilon}$ is a bundle isomorphism between $\xi_{\varepsilon}$ and $\xi_0$.
\end{itemize}
\end{theorem}
This theorem is false if $P$ has dimension greater than $3$: a counter-example by Thurston is described in \cite{sullivan} and also, in a slightly modified version, in Section $3$ of this article.

The hypothesis of existence of a smooth curve of Seifert's leaves, which is fundamental in the previous theorem, is satified when the perturbation $\varepsilon \rightarrow \mathcal{F}_{\varepsilon}$ is real analytic and the base space is not a torus, hence leading to
\begin{corollary}\label{analitico}
If in the previous theorem we substitute the hypothesis of existence of a smooth curve of Seifert's leaves with

(i) $\varepsilon \rightarrow \mathcal{F}_{\varepsilon}$ is real analytic and the Euler characteristic of the base space of $\xi_0$ satisfies $\chi (M)\neq 0$

the conclusions of the Theorem (\ref{main}) still hold, with $\varepsilon \rightarrow \varphi_{\varepsilon}$ real analytic family of real analytic diffeomorphisms. 
\end{corollary}

Theorem (\ref{main}) stem from the following definition and theorem
\begin{definition}\label{tangenza}
Let $\mathcal{F}_0$, $\mathcal{F}$ be two foliations by circles on $P$: they are {\it tangent} at $\gamma_0$ if $\gamma_0$ is a common leaf of $\mathcal{F}_0$ and $\mathcal{F}$. If $\varepsilon \rightarrow \mathcal{F}_{\varepsilon}$ is, as in Definition (\ref{perturbazione}), a smooth deformation of the foliation $\mathcal{F}_0$, whose leaves are the fibers of the circle bundle $\xi_0$ defined in (\ref{bundle}) and $\gamma_0$ is a common leaf of all $\mathcal{F}_{\varepsilon}$, we will say that $\varepsilon \rightarrow \mathcal{F}_{\varepsilon}$ is {\it tangent} to $\mathcal{F}_0$ at $\gamma_0$.
\end{definition}
\begin{theorem}\label{parallele}
Let $\varepsilon \rightarrow \mathcal{F}_{\varepsilon}$ be as in Definition (\ref{perturbazione}) a smooth deformation of the foliation $\mathcal{F}_0$ whose leaves are the fibers of the circle bundle $\xi_0$ defined in (\ref{bundle}), with $P$, $M$ closed oriented manifolds, $dim P=3$. If $\varepsilon \rightarrow \mathcal{F}_{\varepsilon}$ is tangent to $\mathcal{F}_0$ at $\gamma_0$, there exists a smooth family of bundle isomorphisms $\varepsilon \rightarrow \varphi_{\varepsilon}$, isotopic to the identity, which conjugates each $\mathcal{F}_{\varepsilon}$ to $\mathcal{F}_0$.
\end{theorem}

The proof of this theorem, and those of Theorem (\ref{main}) and Corollary (\ref{analitico}) which easily follows from it, will be given in Section $2$, while in Section $3$ we will discuss the hypotheses on which these results are based. In the rest of this section, firstly we will describe an example where all the main arguments entering in the proof of Theorem (\ref{parallele}) will be easily recognizable, and later we will conclude giving a short account of the dynamical problem which motivates these investigations.

{\bf Example}: {\it perturbations of the Hopf bundle which are tangent to a distinguished fiber}.

Let $\underline x \in \mathbb{R}^4$, let $\Vert \cdot \Vert$ be the Euclidean norm and
\begin{equation}\label{hopf}
\dot {\underline x} = A \underline x = X_0 (\underline x)
\end{equation}
where
\[
A=
\left(
\begin{array}{cccc}
0 & -1 & 0 & 0  \\
1 & 0 & 0 & 0 \\
0 & 0 & 0 & -1 \\
0 & 0 & 1 & 0 
\end{array} \right).
\]
$A=-A^t$ hence $S^3 = \{ \underline x : \Vert \underline x \Vert^2 =1 \}$ is invariant for (\ref{hopf}) therefore
$$
X_0 : S^3 \rightarrow T S^3
$$
is the dynamical system defined by two identical harmonic oscillators constrained to the $1$-energy level.
Writing (\ref{hopf}) as
\[
\begin{cases}
\dot z_1 & = i z_1 \\
\dot z_2 & = i z_2
\end{cases}
\]
$z_1 = x_1 + i x_2$, $z_2 = x_3 + i x_4$, and defining $\pi (z_1 , z_2) =[z_1:z_2] \in \mathbb{CP}$, $X_0$ turns to be the isochronous infinitesimal generator of the Hopf circle bundle
$$
\xi_0 \, : \, S^1 \hookrightarrow S^3 \rightarrow^{\pi} \mathbb{CP}\simeq S^2.
$$
The fibers of this bundle define a foliation by circles $\mathcal{F}_0$ of $S^3$. We will consider a smooth deformation
$$
\varepsilon \rightarrow X_{\varepsilon}
$$
of $X_0$, where each
$$
X_{\varepsilon} : S^3 \rightarrow T S^3
$$
define a foliation by circles $\mathcal{F}_{\varepsilon}$ of $S^3$. Moreover we suppose that there exists a common leaf $\gamma_0$ for every $\mathcal{F}_{\varepsilon}$, and along this leaf ${X_0}_{\vert \gamma_0} = {X_{\varepsilon}}_{\vert \gamma_0}$. In other words, $\varepsilon \rightarrow \mathcal{F}_{\varepsilon}$ is {\it tangent} to $\mathcal{F}_0$ {\it at} $\gamma_0$: the equality of the restriction of the $X_{\varepsilon}$'s at $\gamma_0$ can always be obtained by suitable reparametrization. We will show in the next section that this assumption is equivalent to the existence of a smooth curve of Seifert's leaves in the perturbation, {\it cfr.} Lemma (\ref{tangenza}).

Without loss of generality we can suppose that $\gamma_0$ is the trajectory of (\ref{hopf}) through the north pole $N=(0,0,0,1)$ of the $3$-sphere. Let
\[
\begin{cases}
\varphi &:S^3 -\{ N \} \rightarrow \mathbb{R}^3 \\
\underline y & =\varphi (\underline x) = \frac{\underline x'}{1-x_4}
\end{cases}
\]
$\underline x'=(x_1 , x_2 , x_3)^t$, be the stereographic map from north pole. Denoting $\varphi_* X_0 = d \varphi (\varphi^{-1})X_0 (\varphi^{-1})$, the Hopf vector field (\ref{hopf}) in $\underline y$-coordinates is
\[
\dot {\underline y} = (\varphi_* X_0)(\underline y)
\]
or more explicitly
\begin{equation}\label{stereo}
X_0 :
\begin{cases}
\dot y_1 & = -y_2 +y_1 y_3 \\
\dot y_2  &= y_1 + y_2 y_3 \\
\dot y_3 & = \frac{1}{2}(1 + y_3^2 -(y_1^2 + y_2^2)).
\end{cases}
\end{equation}
We observe that there exists a surface $\Sigma$, a ramified double covering of a $2$-sphere, with branching points at $\gamma_0$,  which is a {\it quasi-section} of $X_0$: the definition of this object follows, but we wish to warn the reader that it will be never used in statements and proofs of this article, and is reported here only to help the geometric insight into the problem. 
\begin{definition}\label{sezione}
Let $X_0 :P \rightarrow TP$ be a vector field on a manifold $P$. A {\it quasi-section} $\Sigma$ of $X_0$ is a codimension $1$ topological submanifold of $P$, possibly with boundary, which intersects all the trajectories of $X_0$, this intersection being transverse except at a submanifold $S$ of $\Sigma$ fibered by trajectories of $X_0$. Moreover $\Sigma$ is a smooth ramified covering over the orbit space $M$, with branch points at $S$. If $\Sigma$ has boundary this should be union of finitely many trajectories of $X_0$.
\end{definition}
Of course, the notion of quasi-section is close to that of Birkhoff section: a Birkhoff section is always a quasi-section, but a quasi-section (as will be our concern) need not to have boundary. On the other hand, usually after a suitable cutting of a quasi-section $\Sigma$ along the ramification set $S$, one obtains a Birkhoff section.

In the case of our example, in stereographic coordinates the quasi-section is
$$
\Sigma = \{ y_2 = 0 \} \cup \{ N \} 
$$
which is a two-sheeted ramified covering of the base space $S^2$, with branch points at $\gamma_0$. In fact, in the stereographic chart
\[
{X_0}_{\vert \Sigma} :
\begin{cases}
\dot y_1 &=y_1 y_3 \\
\dot y_2 &=y_1 \\
\dot y_3 &=\frac{1}{2} (1 + y_3^2 - y_1^2)
\end{cases}
\]
and the scalar product between this vector field and the unit normal vector to $\Sigma$
\[
\underline n =
\left(
\begin{array}{c}
0  \\
1 \\
0 
\end{array} \right)
\]
satifies
\[
{X_0}_{\vert \Sigma} \cdot \underline n = y_1
\]
which is always non-zero on $\Sigma - \gamma_0 = \Sigma - \{ y_1 = y_2 =0 \}$.

 To resolve the contact between the Hopf bundle and $\Sigma$ we blow up the fiber bundle $\xi_0$ along $\gamma_0$, a particular instance of a classical construction which will be described in detail in the next section. In stereographic coordinates, and in one of the two coordinates charts covering a neighbourhood of the divisor of the blown up manifold, this amounts to changing coordinates according to
\begin{equation}\label{blowup}
\sigma :
\begin{cases}
y_1 &=y_1 \\
u &=\frac{y_2}{y_1} \\
y_3 &=y_3
\end{cases}
\end{equation}
hence obtaining the lifting of $X_0$ to the vector field $\tilde X_0$ given by
\begin{equation}\label{scoppiato}
 \tilde X_0 :
\begin{cases}
\dot y_1 &= y_1(-u+y_3) \\
\dot u &= 1+u^2 \\
\dot y_3 &= \frac{1}{2}((1+y_3^2)-y_1^2(1+u^2)).
\end{cases}
\end{equation}

This vector field is real analytic up to the divisor $\mathcal{E}= \mathbb{RP} \times S^1 = Klein \, bottle$, whose equation in local coordinates is
$$
\mathcal{E}= \{ y_1 = 0 \} \cup \{ point \, at \, infinity \, of \, \mathbb{RP} \} \times S^1
$$
where $S^1 = \{y_3 - axis   \} \cup \{ N  \}$.

The blow up of the total space $S^3$ of the Hopf bundle is a closed Seifert manifold $\tilde S^3$, and the integral curves of $\tilde X_0$ define a Seifert fibration of $\tilde S^3$, according to the slight generalization of the original Seifert's definition given by Scott \cite{scott}. In fact, all the fibers in $\tilde S^3 - \mathcal{E}$ are regular, while a neighbourhood of $\mathcal{E}$ in $\tilde S^3$ is a solid Klein bottle. The blow up $\tilde \Sigma$ of $\Sigma$, in the local chart described by (\ref{blowup}), is
$$
\tilde \Sigma= \{ u y_1 =0 \} = \{ u=0 \} \cup \mathcal{E} = \hat \Sigma \cup \mathcal{E}
$$
where in local coordinates $\hat \Sigma = \{ u=0 \}$ and $\hat \Sigma \pitchfork \mathcal{E}$: here $\tilde \Sigma$ is the {\it total transform} and $\hat \Sigma$ is the {\it strict} or {\it proper transform} of $\Sigma$. From (\ref{scoppiato}) in local coordinates
\begin{equation}\label{divisore}
\tilde {X_0}_{\vert \mathcal{E}} :
\begin{cases}
\dot u & =1 + u^2 \\
\dot y_3 &= \frac{1}{2}(1+y_3^2)
\end{cases}
\end{equation}
{\it i.e.} the divisor $\mathcal{E} \simeq \mathbb{RP} \times S^1$ is a $2$-torus which is fibered by the integral curves of $\tilde X_0$ in closed curves of homotopy type $(2,1)$. The following properties easily follows, {\it cfr.} Proposition (\ref{risolvi}) :
\begin{itemize}
\item [(i)] all the closed orbits of $\tilde X_0$ meet transversally $\hat \Sigma$, {\it cfr.} (\ref{divisore}), hence:

{\it blowing up the Hopf bundle along a fiber we obtain a Seifert manifold, with a Seifert fibration defined by the integral curves of $\tilde X_0$, and $\hat \Sigma$ is a global section of $\tilde X_0$}, {\it i.e.} $\hat \Sigma$ intersects transversally every integral curve of $\tilde X_0$ \\
\item [(ii)] tangency of $\varepsilon \rightarrow \mathcal{F}_{\varepsilon}$ to $\mathcal{F}_0$ along $\gamma_0$ implies that $\tilde {X_{\varepsilon}}_{\vert \mathcal{E}} =\tilde {X_0}_{\vert \mathcal{E}}$, and this property togheter with compactness of $\tilde S^3$ implies that for sufficiently small $\overline \varepsilon >0$, and for every $\varepsilon$, $\vert \varepsilon \vert < \overline \varepsilon$, {\it every $\tilde X_{\varepsilon}$ has $\hat \Sigma$ as global section}, too \\
\item [(iii)] the dynamics of the vector fields $\tilde X_{\varepsilon}$ on the divisor $\mathcal{E}$  define a foliation of $\mathcal{E}$ by circles of homotopy type $(2, 1)$, $E=1$ being the Euler number of the Hopf bundle.
\end{itemize}
From these properties, in particular from $(ii)$, is easy to deduce, {\it cfr.} next section, that the foliations $\mathcal{F}_{\varepsilon}$ are generated by circle bundles which are smoothly isomorphic to $\xi_0$, {\it i.e.} Theorem (\ref{main}).

The main point of Section $2$ is to generalize $(i), (ii), (iii)$ to any circle bundle on $3$-manifolds, a crucial point being a precise statement of a well-known {\it localization principle} allowing to concentrate the topological information about a circle bundle over a closed oriented surface in an arbitrarily small neighbourhood of one of its fibers, see Lemma (\ref{localizzazione}) and \cite{montesinos} $\S 1$.

We end this introduction adding few words about the dynamical problem which originated this investigation.

\begin{definition}
An {\it oscillator} is a couple $(P,\mathcal{F})$, where $P$ is a closed manifold and $\mathcal{F}$ is a foliation by circles of $P$.

\end{definition}
We will usually refer, when this will be possible and will cause no ambiguities, to an oscillator $(P,\mathcal{F})$ as a couple $(P,X)$ where $X:P \rightarrow TP$ is a vector field whose integral curves are the leaves of $\mathcal{F}$.

The simplest example of an oscillator is $(S^1, \frac{\partial}{\partial \theta})$, another example is the Hopf bundle $(S^3 , X_0)$ defined in (\ref{hopf}). In general, circle bundles form a very special class of oscillators, which have a kind of {\it characteristic frequencies}, namely the integers numbers related to the characteristic class of the circle bundle when a basis on $H^2(M, \mathbb{Z})$, $M$ base space of the bundle, has been fixed. We are interested to the following questions:
\begin{itemize}
\item[(I)] Let $\xi_0$, $\xi_1$ be two circle bundles with total space $P$ and not homeomorphic orbit spaces $M_0$, $M_1$. Is it possible to construct a smooth $1$-parameter family $\varepsilon \rightarrow (P, \mathcal{F}_{\varepsilon})$ of oscillators such that $\mathcal{F}_{0}$, respectively $\mathcal{F}_1$, is the foliation whose leaves are the fibers of $\xi_0$, respectively $\xi_1$? \\
\item [(II)] is it possible to construct a smooth $1$-parameter family of oscillators connecting two circle bundles with the same total space and over the same base space, having different characteristic classes?
\end{itemize}

Theorem (\ref{main}) and Corollary (\ref{analitico}) answer, under suitable hypotheses, in the negative to these questions when $dim P=3$, while the example discussed in Section $3$ answers in the affirmative to question $(I)$ when $dim P=4$ and $\chi (M)=0$.

{\it Aknowledgements}: some of the arguments in this article are reminiscent of conversations I had several years ago with the late Marco Brunella.

\section{Blowing up of a circle bundle along a fiber and existence of a quasi-section}

In this section we will prove Theorem (\ref{parallele}), and its consequences Theorem (\ref{main}) and Corollary (\ref{analitico}), following the scheme sketched in the example developed in the previous section. In this framework, the first step will be to reduce the hypothesis of existence of a smooth curve of Seifert's leaves of $\varepsilon \rightarrow \mathcal{F}_{\varepsilon}$ converging to $\gamma_0$  to the condition of tangency of such perturbation to $\mathcal{F}_0$ along $\gamma_0$, according to Definition (\ref{tangenza}). This is the content of
\begin{lemma}\label{tangenza}
If $\varepsilon \rightarrow \mathcal{F}_{\varepsilon}$ is a smooth deformation of $\mathcal{F}_0$ as in Definition (\ref{perturbazione}) which has a smooth cuve of Seifert's leaves converging to $\gamma_0$, there exists a smooth family of diffeomorphisms $\varepsilon \rightarrow \eta_{\varepsilon}$, $\eta_{\varepsilon} : P \rightarrow P$, such that $\eta_0 = identity$ and $\eta_{\varepsilon} (\gamma_{\varepsilon}) = \gamma_0$. In other words, $\varepsilon \rightarrow \eta_{\varepsilon}(\mathcal{F}_{\varepsilon})$ is a smooth perturbation of $\mathcal{F}_0$ which is tangent to $\mathcal{F}_0$ along $\gamma_0$.
\end{lemma} 
\begin{proof}
From Definition (\ref{perturbazione}) and basic theory of differential equations, up to reparametrizing the vector fields $X_{\varepsilon}$ as in the remark following Definition (\ref{liscia}) , there exists a neighbourhood $U$ of $q_0=\pi (\gamma_0)$ in the base space $M$ of $\xi_0$ and local trivializing coordinates $(\theta , x , y )$, $\theta = \theta \,  mod 2  \, \pi$, in $S^1 \times U$ near $\gamma_0$ such that $\psi : S^1 \times U \rightarrow  \pi^{-1}(U)$ is the equivariant trivializing  diffeomorphism and the foliations $\mathcal{F}_{\varepsilon}$ in trivializing coordinates are described by
\[
X_{\varepsilon} :
\begin{cases}
\dot \theta & = 1 \\
\dot x & = \varepsilon X(x,y,\theta , \varepsilon) \\
\dot y &= \varepsilon Y(x,y,\theta ,\varepsilon)
\end{cases}
\]
where the functions entering in the definition of this differential equation are smooth, $\gamma_0 = \{ (\theta , x , y )  : x=y=0 \}$ and the smooth curve of Seifert's leaves can be represented as
\begin{equation}\label{curva}
\gamma_{\varepsilon} : (\varepsilon , \theta )  \rightarrow (\theta,x(\theta , \varepsilon),y(\theta , \varepsilon))
\end{equation}
where $\theta   \rightarrow (x(\theta , \varepsilon)$, $\theta   \rightarrow (y(\theta , \varepsilon)$ are $2 \pi$-periodic  smooth functions for every $\varepsilon$. Let
\[
W_{\varepsilon} :
\begin{cases}
\dot \theta & = 0 \\
\dot x & = -x(\theta , \varepsilon) \\
\dot y &= -y(\theta , \varepsilon)
\end{cases}
\]
and let $\rho : P \rightarrow \mathbb{R}$ be a smooth function, $0\leq \rho \leq 1$ with support in  $\pi^{-1}(U)$, such that $\rho \equiv 1$ in $\pi^{-1}(U')$,   $U' \subset \subset U$. Let
\[
Z_{\varepsilon} = \rho W_{\varepsilon}.
\]

The flow ${\phi^t}_{\varepsilon}$ defined by this vector field is a $1$-parameter group of automorphisms of $P$ which is the identity outside $\pi^{-1}(U)$ and which in $\pi^{-1}(U')$ can be written in local trivializing coordinates as
\[
{\phi^t}_{\varepsilon} (\theta , x(\theta , \varepsilon) , y(\theta , \varepsilon))= (\theta , -x(\theta , \varepsilon) t + x(\theta , \varepsilon) , -y(\theta , \varepsilon) t + y(\theta , \varepsilon)).
\]
Therefore
\[
\eta_{\varepsilon} = \phi^1_{\varepsilon}
\]
satisfies the statement.
\end{proof}
Up to substituting $\varepsilon \rightarrow \mathcal{F}_{\varepsilon}$ with $\varepsilon \rightarrow \eta_{\varepsilon}(\mathcal{F}_{\varepsilon})$, Theorem (\ref{main}) reduces to Theorem (\ref{parallele}), whose proof will fill most part of this section.

The next lemma formalizes a well-known localization principle, stating that the all the information needed to classify up to smooth bundle equivalence a circle bundle $\xi_0$ of type (\ref{bundle}), with total space $P$ which is a closed oriented $3$-manifold and base space which is a closed oriented surface, is concentranted in an arbitrarily small neighbourhood of a distinguished fiber $\gamma_0$, and reduces to an integer $E$, the {\it Euler number} of $\xi_0$. We also need to explain $E$ in terms of the variational equation of $X_0$ along $\gamma_0$, or equivalently as linking number between $\gamma_0$ and a fiber of the bundle in a neighbourhood of $\gamma_0$. Though most, if not all, these iussues are known, {\it cfr.} \cite{montesinos} $\S 1$ for a qualitative account and \cite{bott&tu} $\S 11$ for some analytic proofs in a setting similar to that of the present article, we choose to give a complete proof of Lemma (\ref{localizzazione}), for it stays at the core of the proofs of our main results.

Let $\xi_0$ as in (\ref{bundle}), with $dim P = 3$, $P, M$ both closed and oriented, and let $\gamma_0$ be a distinguished fiber of $\xi_0$, $q_0 = \pi (\gamma_0)$. It will be useful to choose on $M$ a Riemannian metric, and denote $\mathbb{D}_r (q_0)$ the disc centered at $q_0$ of radius $r$. Let $U_0 = \mathbb{D}_{2r} (q_0)$, $U_1 = M-\mathbb{D}_{\frac{r}{2}} (q_0)$: hence
\[
\partial \mathbb{D}_r (q_0) \subset U_1 \cap U_0.
\]
Along the proof of Lemma (\ref{localizzazione}) we will prove that a bundle $\xi_0$ of the above described characteristics, with a distinguished fiber $\gamma_0$, always admits a bundle structure made by the cover $\mathcal{U}=\{U_0 , U_1 \}$ and by the trivializing diffeomorphisms
\begin{equation}\label{trivializzazione}
\begin{cases}
\psi_0 & : \pi^{-1} (U_0) \rightarrow U_0 \times S^1 \\
\psi_1 & : \pi^{-1} (U_1) \rightarrow U_1 \times S^1
\end{cases}
\end{equation}
which are $S^1$-equivariant, the action on $P$ being that generated by $X_0$ and the action on the product spaces being the natural action on the second factor. We fix a complex coordinate $w$ on $U_0$, and refer $U_0 \cap U_1$ to such coordinate, too. We will denote
\[
\theta = \frac{w}{\vert w \vert}
\]
and $z_0$, respectively $z_1$, the fiber coordinate on $U_0 \times S^1$, respectively on $U_1 \times S^1$. We will also denote $z_0 = e^{i \varphi_0}$, $z_1 = e^{i \varphi_1}$. Th transiction function
\[
\begin{cases}
g_{01} &: U_1 \cap U_0 \rightarrow S^1 \\
z_0 &= g_{01} z_1
\end{cases}
\]
completely characterizes $\xi_0$. The above fixed notations will be used in the statement of the following lemma, whose proof is almost completely contained in \cite{montesinos}, \cite{bott&tu}, \cite{constantin}.

\begin{lemma}\label{localizzazione}
\begin{itemize}
\item [(i)] there exists a smooth section $s : U_1 \rightarrow \pi^{-1}(U_1)$ \\
\item [(ii)] $\xi_0$, with the distinguished fiber $\gamma_0$, admits the trivializing structure (\ref{trivializzazione}), with transiction function
\[
g_{01}(w)=(\frac{w}{\vert w \vert})^E \\
\]
\item [(iii)] the variational equation along $\gamma_0$ of $X_0$ is
\[
\frac{\partial X_0}{\partial w} :
\begin{cases}
\dot w &= -i E w \\
\dot \varphi_0 &= 1
\end{cases}
\]
\\
\item [(iii)'] there exist good coordinates, still named $(w , \varphi_0)$, in a neighbourhood of $\gamma_0$ in $P$ such that

\[
X_0 :
\begin{cases}
\dot w &= -i E w \\
\dot \varphi_0 &= 1
\end{cases}
\]
\\
\item [(iv)]  $E$ is the linking number in $P$ between $\gamma_0$ and a fiber of $\xi_0$ in a sufficiently small neighbourhood of $\gamma_0$.
\end{itemize}
\end{lemma}
\begin{proof}
 Let $g$ be the genus of the base space $M$: a direct application of the Handle Presentation Theorem, see \cite{kosinski} for the general theory, and \cite{montesinos} $\S 1$, \cite{constantin} for our case of study, gives the decomposition
\begin{equation}\label{manici}
M=H^0 \bigsqcup 2 g H^1 \bigsqcup H^2
\end{equation}
where
\[
\begin{cases}
H^0 &= \{ point \} \times \overline {\mathbb{D}}^2 \\
H^1 &= \overline {\mathbb{D}}^1 \times \overline {\mathbb{D}}^1 = [0,1] \times [0,1] \\
H^2 &= \overline {\mathbb{D}}^2 \times \{ q_0 \}
\end{cases}
\]
where the $j$-handles $H^j$, $j=0,1,2$, are joined through smooth glueing maps (which we usually forget in the description which follows) which identify points of the boundary of of each $H^1$ with portions of the boundary of $H^0$ or $H^1$. Moreover in (\ref{manici})  we write $2g H^1$  for $2g$ copies of the $1$-handles. One important point is that the Classification Theorem for compact $1$-manifolds \cite{milnor} implies that the glueing maps are unique (up to orientation preserving diffeomorphisms), and therefore (\ref{manici}) contains all the information about $M$. We choose $H^2$ such that
\[
H^2 = \overline {\mathbb{D}_{2r}(q_0)} = \overline {U_0}
\]
and
\[
\overline U_1 = H^0 \bigsqcup 2g H^1.
\]
We proceed now to construct the section
\[
s : U_1 \rightarrow \pi^{-1}(U_1)
\]
following the unpublished dissertation \cite{constantin}. $H^0$ is contractible, hence a section $s : H^0 \rightarrow \pi^{-1}(H^0)$ exists. We whish to extend it to $s: H^0 \bigsqcup H^1 \rightarrow \pi^{-1}(H^0 \bigsqcup H^1)$. Once this will be done, the proof of statement $(i)$ will follow from the application of this same argument $2g$-times. The passage from the section over $H^0$ to the section over $H^0 \bigsqcup H^1$ goes as follows. The restriction of the section $s : H^0 \rightarrow \pi^{-1}(H^0)$ to $H^0 \cap \partial H^1$, where $\partial H^1 = [0,1] \times \{ 0 \} \cup  [0,1] \times \{ 1 \}$ reduces the problem to the interpolation between the section $s$ over $[0,1] \times \{ 0 \}$ and over $ [0,1] \times \{ 1 \}$. There are several ways to realize this interpolation, but perhaps the easier one, given by Constantin in \cite{constantin}, is the following. Let
\[
\begin{cases}
\rho &: \mathbb{R} \rightarrow S^1 \simeq \frac{\mathbb{R}}{2 \pi \mathbb{Z}} \\
\rho (t)& = e^{it}
\end{cases}
\]
the universal covering map of the circle, and lift $s_j: I \times \{ j \} \rightarrow \pi^{-1} (I \times \{ j \}) \simeq I \times \{ j \} \times S^1$, $j=0,1$, to
\[
\hat {s_j}: I \times \{ j \} \rightarrow \pi^{-1} (I \times \{ j \}) \simeq I \times \{ j \} \times \mathbb{R}.
\]
We interpolate linearly the $\hat {s_j}$'s, $j=0,1$ to give
\[
\tilde s (\tau , u) = (1-\tau)\hat {s_0} (u) + \tau \hat {s_1} (u)
\]
where $\tau , u \in [0,1]$ and then define $s:H^0 \bigsqcup H^1 \rightarrow \pi^{-1}(H^0 \bigsqcup H^1)$ as the section $s$ over $H^0$ and extend it as $\rho \circ \tilde s (\tau ,u)$ if $(\tau , u) \in H^1$, obtaining the desired smooth section over $H^0 \bigsqcup H^1$. The proof of statement $(i)$ is concluded.

To proceed to the proof of statement $(ii)$ we need the following slight variation to a result proved in \cite{bott&tu} $\S 11$ for rank $2$ vector bundles and with referement to a trivializing neighbourhood of $\gamma_0$

\begin{lemma}\label{grado}
Let
\[
\rho : N_{\gamma_0} \rightarrow \gamma_0
\]
be the normal bundle to $\gamma_0$ in $P$, and let $N_{\gamma_0}^{\varepsilon}$ be a sufficiently small neighbourhood of the zero section which is diffeomorphic to a small tubular neighbourhood of $\gamma_0$ in $P$. Let $s :U_1 \rightarrow \pi^{-1}(U_1)$ be the section of $\xi_0$ whose existence has been proved in statement $(i)$. Then, for sufficiently smooth $r>0$
\[
\rho \circ s_{\vert} : \partial \mathbb{D}_r (q_0) \simeq S^1 \rightarrow \gamma_0 \simeq S^1
\]
and
\[
deg \, \rho \circ s_{\vert} =E .
\]
\end{lemma}
\begin{proof} (of Lemma (\ref{grado})) we recall that the Euler class $e \in H^2 (M, \mathbb{Z})$ is the image of the isomorphism between the abelian group of circle bundles and $H^2 (M, \mathbb{Z})$, {\it cfr.} \cite{chern}: this diffeomorphism depends on the chosen covering of $M$, {\it i.e.} it essentially depends on the choice of $\gamma_0$. The Euler number when $dim M=2$ is
\[
E=\int_{M} e .
\]
Let $\psi$ be a connection $1$-form of $\xi_0$, {\it i.e.} a globally defined $1$-form on $P$ which in the two trivializing charts (\ref{trivializzazione}) has the form \cite{chern}
\begin{equation}\label{conn}
\begin{cases}
\psi = - d\varphi_0 + \pi^* \theta_{U_0} \, \, in \, \, U_0 \\
\psi = - d\varphi_1 + \pi^* \theta_{U_1} \, \, in \, \, U_1 .
\end{cases}
\end{equation}
The gauge potentials $U_0 , U_1$, though only locally defined, satisfy
\begin{equation}\label{gauge}
\theta_{U_1} - \theta_{U_0} = i \, d \, log \, g_{01} \, \, in \, \, U_0 \cap U_1
\end{equation}
hence
\[
\Omega = d \theta_{U_1} = d \theta_{U_0}
\]
is a globally defined $2$-form in $M$, the curvature of the connection, and
\[
d \psi = \pi^* \Omega.
\]
From (\ref{gauge}) it easily follows that if $\tilde {\theta_{U_0}} , \tilde {\theta_{U_1}}$ is another collection of gauge potentials then
\[
\theta_{U_0} - \tilde {\theta_{U_0}} = \theta_{U_1} - \tilde {\theta_{U_1}} = \chi
\]
is a globally defined $1$-form on $M$, therefore for any connection the curvature $\Omega$ defines the same cohomology class of $e$ hence
\[
\int_M \Omega = E.
\]

Let $\rho_0 , \rho_1$ be a partition of unity subordinate to the cover $\mathcal{U} = \{   U_0 , U_1 \}$, then defining
\[
\begin{cases}
\theta_{U_0} &= \frac{1}{2 \pi}\rho_0 \, i \, d \, log \, g_{01} \\
\theta_{U_1} &= \frac{1}{2 \pi}\rho_1 \, i \, d \, log \, g_{01}
\end{cases}
\]
we get that
\begin{equation}\label{supporto}
supp \, \Omega \subset U_0 \cap U_1 .
\end{equation}
Recalling that $\pi \circ s = \Id_{U_1}$, we obtain
\[
E=\int_M \Omega = \int_{M-\mathbb{D}r(q_0)} s^* \pi^* \Omega = \int_{M-\mathbb{D}r(q_0)} s^* d \psi = - \int_{\partial  \mathbb{D}_r (q_0)} s^* \psi
\]
where in the last equality we used Stokes' Theorem and considerations on the orientation of $\partial  \mathbb{D}_r (q_0)$ as boundary of $\mathbb{D}r(q_0)$ or $M.-\mathbb{D}r(q_0)$. On the other hand, from (\ref{supporto}) we have that over $\partial  \mathbb{D}_r (q_0)$
\[
\psi = - d \varphi_0
\]
{\it i.e.} on $\partial  \mathbb{D}_r (q_0)$
\[
-\psi = \rho^* \sigma
\]
where $\sigma$ is the generator of the $1$-dimensional cohomology on $\gamma_0$, therefore
\[
E= - \int_{\partial  \mathbb{D}_r (q_0)} s^* \psi = \int_{\partial  \mathbb{D}_r (q_0)} s^* \rho^* \sigma = \int_{\partial  \mathbb{D}_r (q_0)} (\rho \circ s)^* \sigma = deg (\rho \circ s)
\]
and this ends the proof of Lemma (\ref{grado}).
\end{proof}
We come back now to the proof of statement $(ii)$ of Lemma (\ref{localizzazione}). We observe that
\[
{g_{01}}_{\vert} : \partial \mathbb{D}_r (q_0) \simeq  S^1 \rightarrow \gamma_0 \simeq S^1 
\]
has a well-defined degree that in view of Lemma (\ref{grado}) must satisfy
\[
deg \, {g_{01}}_{\vert} = E
\]
which proves the second statement of Lemma (\ref{localizzazione}).

We come now to the proof of statements $(iii), (iii)'$. Firstly, we prove that these two statements are equivalent, as a consequence of the following well-known

\begin{lemma}\label{sospensione}
Let $X: \mathcal{T} \rightarrow T \mathcal{T}$ be a smooth vector field on the $3$-dimensional solid torus $\mathcal{T} = \mathbb{D}_r \times S^1$, whose integral curves define a foliation by circles of $\mathcal{T}$. We suppose that $\gamma_0 = \{ 0 \} \times S^1$ is an integral curve of $X$ and that the period function of $X$ is smooth. Then, up to choosing $r>0$ sufficiently small, $X$ is smoothly conjugated to the vector field
\begin{equation}\label{bochner}
X_L :
\begin{cases}
\dot w &= i l w \\
\dot \varphi &= 1
\end{cases}
\end{equation} 
where $w \in \mathbb{C}$, $\vert w \vert < r$, is a complex coordinate in $\mathbb{D}_r$, $\varphi = \varphi \, mod \, 2 \pi$ is a coordinate along $\gamma_0$ and $l \in \mathbb{Z}$.
\end{lemma}
\begin{proof} ( of Lemma (\ref{sospensione}))
Firstly, we observe that up to smooth reparametrization obtained through multiplication of $X$ by a smooth positive function, we can suppose that all the integral curves of $X$ are $2 \pi$-periodic. Up to reducing $r>0$ $\Sigma = \mathbb{D}_r \times \{ 0 \}$ is a Poincar\'e section of $X$ and $\mathcal{P} : \Sigma \rightarrow \Sigma$ is the relative Poincar\'e map, which is periodic of period $\vert l \vert$. An application of Bochner Linearization Theorem gives
\begin{equation}\label{rotazione}
\zeta \circ \mathcal{P} = R_l \circ \zeta
\end{equation}
where
\[
\begin{cases}
R_l w &= e^{il \theta} \\
\zeta &= \frac{1}{\vert l \vert} \sum_{k=0}^{\vert l \vert -1} R_l^{-k} \circ \mathcal{P}^k  
\end{cases}
\]
where $\theta = \frac{w}{\vert w \vert}$. The conjugacy $\zeta$ can be suspended to give a conjucacy between the flows $\phi^t_X , \phi^t_L$ of $X$ and $X_L$ as follows, hence concluding the proof of the lemma. Let $\tau : \mathcal{T} \rightarrow [0,2 \pi[$ the smooth function, defined by the Implicit Function theorem, such that $\phi^{-t}_X \notin \Sigma$ if $0<t<\tau$, $\phi^{--\tau}_X \in \Sigma$. The group property of the flow of $X$ implies that for sufficiently small $t>0$
\begin{equation}\label{gruppo}
\tau (\phi^t_X) = \tau + t .
\end{equation}  
Let
\[
\eta = \phi^{\tau}_L \circ \zeta \circ \phi^{-\tau}_X
\]
then from (\ref{gruppo})
\[
\phi^t_L \circ \eta \circ \phi^{-t}_X = \phi^t_L \circ \phi^{\tau}_L \circ \zeta \circ \phi^{-\tau}_X \circ \phi^{-t}_X = \eta
\]
{\it i.e.} $\phi^t_L \circ \eta = \eta \circ \phi^t_X$.

\end{proof}
So statements $(iii)$ and $(iii)'$ are equivalent and we just need to prove the first one. We also observe that from the proof of the above lemma it follows that the variational equation of $X_0$ in a sufficiently small neighbourhood of $\gamma_0$ is of the type
\[
\frac{\partial X_0}{\partial w} :
\begin{cases}
\dot w &= i \alpha w \\
\dot \varphi_0 &= 1 
\end{cases}
\]
where $\alpha \in \mathbb{Q}$, therefore we only must prove that $\alpha = - E$. Let us choose, with referement to (\ref{trivializzazione}), on $\mathcal{T} = \psi^{-1}(\partial \mathbb{D}_r (q_0) \times S^1)$ coordinates $(\Theta , \phi_0)$, where $\Theta = \frac{w}{\vert w \vert}$, but these are {\it not} the trivializing coordinates in $U_0 \times S^1$. Firstly, we observe that from
\[
z_0 = g_{01}(w) z_1 = (\frac{w}{\vert w \vert})^E z_1
\]
the foliation of $\mathcal{T}$ defined by $\{ z_1 = constant   \}$ is defined in $(\Theta , \phi_0)$-coordinates by the vector field
\[
X_0^{\perp} :
\begin{cases}
\dot \Theta &= \frac{1}{E} \\
\dot \phi_0 &= 1
\end{cases}
\]
{\it i.e.} naming $m=\{(\Theta , \phi_0): \phi_0 =0    \}$ and $l=\{  (\Theta , \phi_0): \Theta =0      \}$ the integral curves of $X_0^{\perp}$ are in the homotopy class $m + E l$ of $\mathcal{T}$. Let $Z : \mathcal{T} \rightarrow T \mathcal{T}$ be the vector field
\[
Z:
\begin{cases}
\dot \Theta &= -E \\
\dot \phi_0 &=1 .
\end{cases}
\]
$Z$ satisfies
\[
\begin{cases}
Z \pitchfork X_0^{\perp} \\
[Z, X_0^{\perp}]=0
\end{cases}
\]
therefore $Z$ is conjugated on $\mathcal{T}$ to the $S^1$-action which defines the integral curves of $X_0^{\perp}$ starting from one of them as initial data. In other words, $Z$ is conjugated to $X_0$ on $\mathcal{T}$. We observe that the integral curves of $Z$ are of homotopy type $-E m + l$. Extending $Z=X_0$ to $\vert w \vert = r'$, $0<r' <r$, we have that $\vert w \vert$ is a first integral and
\[
\dot w = \frac{d \vert w \vert e^{i \theta}}{dt}= \vert w \vert i e^{i \Theta} \dot \theta = -i E w
\]
and therefore in a sufficiently small neighbourhood of $\gamma_0$ the vector field $X_0$ is
\begin{equation}\label{ultimo}
\begin{cases}
\dot w &= -i E w \\
\dot \phi_0 &= 1 .
\end{cases}
\end{equation}
Statement $(iv)$ is a straightforward application of the interpretation of the linking number between $\gamma_0$ and an integral curve of (\ref{ultimo}) as the intersection number between $\gamma_0$ and a local section to (\ref{ultimo}) transverse to $\gamma_0$.
\end{proof}

Let us suppose that we are able to extend the smooth section over $U_1$ obtained in the previous lemma to a section over the punctured base space $M- \{ q_0 \}$, $s: M- \{ q_0 \} \rightarrow \pi^{-1}(M- \{ q_0 \})$, and let $s(M- \{ q_0 \}) = \Sigma$. One could reasonably conjecture that $\Sigma \cup \gamma_0$ is then a quasi-section, with branch points at $\gamma_0$: in fact, this is what happens to be true in the example we gave in Section $1$, concerning the Hopf fibration. The proof of this claim in full generality has to deal with two difficulties: the proof of the extension up to $M- \{ q_0 \}$ of the previously constucted section (which is not a too serious obstruction: we actually will realize this extension below) and the analytic proof of regularity of the quasi-section up to the boundary.

To avoid such difficulties, and also to clarify the geometry of the tangency of (the closure in $P$ of) $\Sigma$ along the fiber $\gamma_0$ of the foliation defined by $\xi_0$, we abandon the definition of quasi-section and analyze the local behaviour of the foliations $\mathcal{F}_{\varepsilon}$ through the {\it blow up } of $\xi_0$ along the tangency fiber $\gamma_0$.

Though a standard tool in several fields of mathematics, expecially in algebraic geometry, see {\it e.g.} \cite{griffiths-harris}, we find useful to describe the blowing up construction with some details in the particular case of blowing up a fiber $\gamma_0$ of a circle bundle $\xi_0$, proving some properties which will be formally stated in Proposition (\ref{risolvi}).

Blowing up $\xi_0$ at the fiber $\gamma_0$ means the definition of the triple $(\tilde P_{\gamma_0} , \sigma , \tilde {\mathcal{F}}_0)$ where
\begin{itemize}\label{scoppio}
\item [(i)] $\tilde P_{\gamma_0}$ is a smooth manifold \\
\item [(ii)] $\sigma : \tilde P_{\gamma_0} \rightarrow P $ is a smooth map, $\sigma_{\vert} : \tilde P_{\gamma_0} - \mathcal{E} \rightarrow P - \gamma_0$ is a diffeomorphism, $\mathcal{E}=\sigma^{-1}(\gamma_0)$ is the total space of the projectivized normal bundle to $\gamma_0$ \\
\item [(iii)] $\tilde {\mathcal{F}}_0$ is a smooth foliation by circles of $\tilde P_{\gamma_0}$ such that $\sigma_{\vert} (\tilde {\mathcal{F}}_0) = \mathcal{F}_0$ on $ \tilde P_{\gamma_0} - \mathcal{E}$.
\end{itemize}
We recall that the normal bundle to $\gamma_0$ in $P$ is
\[
\begin{cases}
\rho : & N_{ \gamma_0} \rightarrow \gamma_0 \\
N_{ \gamma_0} &= \frac{(TP)_{\vert \gamma_0}}{T\gamma_0}
\end{cases} 
\]
and its projectivization is the projective bundle
$$
\mathcal{E}=\mathbb{P}(N \gamma_0) \rightarrow^{\pi_{\gamma_0}} \gamma_0
$$
whose total space $\mathcal{E}$ is called the {\it divisor} of the blow up.

The construction of $\tilde {P_{\gamma_0}}$ calls for the definition of a smooth atlas $\tilde {\mathcal{A}}$, definition which is based on the atlas $\mathcal{A}$ of $P$ adapted to the bundle structure of $\xi_0$. To clarify this point, we recall that from Lemma (\ref{localizzazione}) $\xi_0$ is trivialized over two sets $U_0 \times S^1$, $U_1 \times S^1$ where $U_0 = \mathbb{D}_{2 r} (q_0)$ and $U_1 = M - \mathbb{D}_{\frac{r}{2}}(q_0)$, $M$ base surface of the bundle. The atlas $\mathcal{A}$ is defined in (\ref{trivializzazione}). A coordinate function is chosen in $U_0$ in the form of a complex coordinate $w=x+iy$, $\vert w \vert < 2r$, $w(q_0)=0$, and we intentionally will identify $U_0$ with the  disc $\mathbb{D}_{2 r} (q_0)$ in $\mathbb{C}$, and $U_0 \cap U_1$ with the punctured disc $\dot {\mathbb{D}_{2 r} (q_0)}$. Moreover we will always consider a coordinate along the fiber $S^1$ as a complex number of modulus $1$, $z=e^{i \varphi}$, or equivalently as a $mod \, 2 \pi$ real variable $\varphi$: in other words, we will always work as this fiber coordinate would be (virtually) a global one.

The atlas $\tilde {\mathcal{A}}$ is obtained by $\mathcal{A}$ removing the local chart
$$
(w,\varphi)  : U_0 \times S^1 \rightarrow \mathbb{D} \times S^1
$$
and substituting it with two charts
\[
\begin{cases}
((x,u),\varphi) &: \tilde {U_{0x}} \times S^1 \rightarrow \{((x,u),\varphi): x^2 (1+u^2)<1\} \times S^1 \\
((v,y),\varphi) &: \tilde {U_{0y}} \times S^1 \rightarrow \{((v,y),\varphi): y^2 (1+v^2)<1\} \times S^1
\end{cases}
\]
the coordinate change between these two charts being
\[
\begin{cases}
v&=\frac{1}{u} \\
y&=xu \\
\varphi &= \varphi.
\end{cases}
\]
All the other coordinate changes are obtained from the atlas $\mathcal{A}$ and the laws (projections obtained by restriction of $\sigma^{-1}$ to the base space)
\[
\begin{cases}
x&=x \\
y&=xu
\end{cases}
\]
and
\[
\begin{cases}
x&=yv \\
y&=y .
\end{cases}
\]
The manifold $\tilde {P_{\gamma_0}}$ has the following geometric description, see \cite{griffiths-harris} $\S 4.6$. Firstly, we observe that the blow up is locally isomorphic to the blow up of the product $\Delta = \mathbb{D} \times I$ where $I$ is an interval parametrizing locally the fiber $S^1$: still naming $\varphi$ the coordinate along $I$ we choose on $\Delta$ local coordinates $(x,y,\varphi)$, and we define
\[
\tilde \Delta \hookrightarrow \Delta \times \mathbb{RP}
\]
as the smooth manifold
\[
\tilde \Delta = \{((x,y,\varphi),l): l=[l_1 , l_2] , x l_2 = y l_1 \}
\]
covered by the two coordinate charts $\tilde {U_{0x}} \times S^1$, $\tilde {U_{0y}} \times S^1$, togheter with the smooth projection
\[
\sigma : \tilde \Delta \rightarrow \mathbb{D} \times S^1
\]
defined in $\tilde {U_{0x}} \times S^1$ by
\[
((x,u,\varphi),[x,y]) \rightarrow ((x,xu,\varphi))
\]
this map satisfying point $(ii)$ in (\ref{scoppio}). In other words, each point along $\gamma_0$ is blown up as a point of the corresponding fiber in the normal bundle to $\gamma_0$ in $P$. The change of coordinates between $(x,u,\varphi)$ and $(v,y,\varphi)$ coordinates proves that $\mathcal{E} = \{ x=0  \}$, resp. $\mathcal{E} = \{ y=0  \}$ in each of these coordinates, and $\mathcal{E} \simeq \mathbb{RP} \times S^1$, {\it i.e.} $\mathcal{E}$ is diffeomorphic to a Klein bottle. This remark ends the description of the first two terms in the triple $(\tilde {P_{\gamma_0}} , \sigma , \tilde {\mathcal{F}_0})$ which forms the blow up of $\xi_0$ along $\gamma_0$.

To lift $\mathcal{F}_0$ to a foliation by circles $\tilde {\mathcal{F}_0}$ of $\tilde {P_{\gamma_0}}$ one could use $\sigma^{-1}$ outside $\gamma_0$ and then try to extended smoothly the obtained foliation by circles on $\tilde {P_{\gamma_0}} - \mathcal{E}$ to the divisor $\mathcal{E}$: as direct as it is this approach is difficult to be carried on. We prefer to lift the infinitesimal law describing $\mathcal{F}_0$: of course, all that really matters is to do that in a neighbourhood of $\gamma_0$.

The differential equation (\ref{ultimo}) can be considered as an equation on the torus $\{ \vert w \vert =r \}$ whose integral curves are the leaves of $\mathcal{F}_0$, and we rewrite it as

\begin{equation}\label{sistema1}
\begin{cases}
\frac{dw}{dt}&=-iE w \\
\frac{dz_0}{dt}&=iz_0
\end{cases}
\end{equation}

Rigorously speaking, equation (\ref{sistema1}), which is a representation in local coordinates of $X_0$, makes sense if $\vert w \vert >0$ only, but it obviouly extend to $w=0$. The blow up transformation
\[
\begin{cases}
x=&x \\
y=&ux \\
\varphi =& \varphi
\end{cases}
\]
transforms the system (\ref{sistema1}), which is equivalent to
\begin{equation}\label{equivalente}
\begin{cases}
\dot x &=  Ey \\
\dot y &=  -Ex \\
\dot \varphi &= 1
\end{cases}
\end{equation}
to
\[
\left(
\begin{array}{c}
\dot x  \\
\dot u \\
\dot \varphi  
\end{array} \right)
=
\frac{\partial (x,u,\varphi)}{\partial (x,y,\varphi)}
\left(
\begin{array}{c}
\dot x  \\
\dot y \\
\dot \varphi 
\end{array} \right)
\]
and therefore computing in $(x,u,\varphi)$ coordinates
\[
\frac{\partial (x,u,\varphi)}{\partial (x,y,\varphi)} = 
\left(
\begin{array}{ccc}
1 & 0 & 0  \\
-\frac{u}{x} & \frac{1}{x} & 0  \\
0 & 0 & 1  
\end{array} \right)
\]
and
\[
\left(
\begin{array}{c}
\dot x  \\
\dot y \\
\dot \varphi
\end{array} \right)
=
\left(
\begin{array}{c}
Eux  \\
-E(1+u^2) \\
1
\end{array} \right)
\]
we get
\begin{equation}\label{lift}
\tilde {X_0} :
\begin{cases}
\dot x &= Exu \\
\dot u &= -E(1+u^2) \\
\dot \varphi &= 1.
\end{cases}
\end{equation}

This vector field, {\it a priori} defined in $\tilde {P_{\gamma_0}} - \mathcal{E}$, extends as a real analytic vector field to the divisor as
\begin{equation}\label{esteso}
\tilde {X_0}_{\vert \mathcal{E}} :
\begin{cases}
\dot x &= 0 \\
\dot u &= -E(1+u^2 )\\
\dot \varphi &= 1
\end{cases}
\end{equation}
and therefore define a foliation by circles $\tilde {\mathcal{F}_0}$ of $\tilde {P_{\gamma_0}}$ which extends to the divisor the lift through $\sigma^{-1}$ of $\mathcal{F}_0$. The restriction of $\tilde {\mathcal{F}_0}$ to the divisor, whose fundamental group is $\mathbb{Z} \oplus \mathbb{Z}$, is a curve of homotopy type $(2E,1)$: while a point of an orbit of $\tilde {X_0}_{\vert \mathcal{E}}$ turns $2E$-times around $\mathbb{RP}$ it winds once around the second factor of $\mathcal{E} \simeq \mathbb{RP} \times S^1$.

A comment is in order: we will always assume that $E \neq 0$, for in this case a global section to $\xi_0$ does exists, and all the statemets we whish to prove are trivially true.

Summarizing, we started by $X_0 : P \rightarrow TP$ defining a circle bundle $\xi_0$ and the relative foliation by circles $\mathcal{F}_0$ and we lifted these mathematical object to the blow up $(\tilde {P_{\gamma_0}},\sigma , \tilde {\mathcal{F}_0})$, $\tilde {X_0} : \tilde {P_{\gamma_0}} \rightarrow T \tilde {P_{\gamma_0}}$: then the couple $(\tilde {P_{\gamma_0}} , \tilde {\mathcal{F}_0})$ defines a Seifert $3$-manifold and a Seifert fibration, in the sense of the definition given by Scott in \cite{scott}: all the fibres of $\tilde {\mathcal{F}_0}$ in $\tilde {P_{\gamma_0}} - \mathcal{E}$ are regular, while a neighbourhood of the divisor $\mathcal{E}$ in $\tilde {P_{\gamma_0}}$ is a solid Klein bottle.

In the next proposition we show that the same extension to the divisor we proved for $\tilde {\mathcal{F}_0}$, $\tilde {X_0}$ holds for the perturbed lifted foliations $\tilde {\mathcal{F}_{\varepsilon}}$ and vector fields $\tilde {X_{\varepsilon}}$, and, most important, that the section defined in Lemma (\ref{localizzazione}) can be lifted and extended to a smooth global section of any $\tilde {X_{\varepsilon}}$ in $\tilde {P_{\gamma_0}}$, $\vert \varepsilon \vert$ sufficiently small.

\begin{proposition}\label{risolvi}
Let $\varepsilon \rightarrow X_{\varepsilon}$ as in Definition (\ref{perturbazione}) and let $\varepsilon \rightarrow \mathcal{F}_{\varepsilon}$ be the associated $1$-parameter family of foliations tangent to $\mathcal{F}_0$ at $\gamma_0$. Let $\Sigma = s (U_1)$ be the section over $U_1$ whose exsistence has been proved in statement $(i)$ of Lemma (\ref{localizzazione}). Then there exists $\overline \varepsilon >0$ such that  for $\vert \varepsilon \vert < \overline \varepsilon$ 
\begin{itemize}
\item [(i)] each $X_{\varepsilon} :P \rightarrow TP$ lifts to a smooth vector field
$$
\tilde X_{\varepsilon} : P_{\gamma_0} \rightarrow T P_{\gamma_0}
$$ 
such that
$$
\tilde {X_{\varepsilon}}_{\vert \mathcal{E}} = \tilde {X_{0}}_{\vert \mathcal{E}} \\
$$ 
\item [(ii)] the strict trasform $\hat \Sigma$ of $\Sigma$ extends to a global section for every $\tilde X_{\varepsilon}$ \\
\item [(iii)] each $\tilde X_{\varepsilon}$ defines a Seifert fibration of $P_{\gamma_0}$: all the fibers of this fibration in $P_{\gamma_0} - \mathcal{E}$ are regular, $\mathcal{E}$ is fibered by the integral curves of $\tilde X_{\varepsilon}$ in closed curves of homotopy type $(2E,1)$, $E$ Euler number of $\xi_0$, and a neighbourhood of $\mathcal{E}$ in $P_{\gamma_0}$ is diffeomorphic to a solid Klein bottle (see \cite{scott}).
\end{itemize}
\end{proposition}

\begin{remark} 
Statement $(iii)$ of the above proposition and the blow up construction described above, suggest the following statement, which is surely known to the experts and whose proof follows by that of the above proposition

{\it 
Each circle bundle, whose total space $P$ is a closed oriented $3$-manifold and whose base space is a closed oriented surface, can be blown up along a fiber $\gamma_0$ to a Seifert manifold $P_{\gamma_0}$. The foliation of $P$ defined by the fibers of the bundle is lifted to a Seifert fibration of $P_{\gamma_0}$ whose fibers are regular except those on the divisor $\mathcal{E}$, a neighbourhood of the divisor being a solid Klein bottle and the foliation along the divisor being that in closed curves of homotopy type $(2E,1)$, $E$ Euler number of the circle bundle. This Seifert fibration has a smooth global section.
}
\end{remark}
\begin{proof} (of Proposition (\ref{risolvi}))

Firstly, we whish to recall that we will always disregard the case $E=0$: in this case the statements of the proposition are easily seen to be trivially true, because a circle bundle with $0$ Euler number has a global section.

The first iussue to tackle is the lift of $X_{\varepsilon}$, and therefore of $\mathcal{F}_{\varepsilon}$, to the blown up manifold. Our starting point is equation (\ref{equivalente}) which leads to the expression in local coordinates for $X_{\varepsilon}$
\begin{equation}\label{infinitesimale}
{X_{\varepsilon}} :
\left(
\begin{array}{c}
\dot x  \\
\dot y \\
\dot \varphi
\end{array} \right)
=
\left(
\begin{array}{ccc}
0&-1&0 \\
1&0&0 \\
0&0&E
\end{array}
\right)
\left(
\begin{array}{c}
x \\
y \\
\varphi
\end{array}
\right)
+ \underline G (x,y,\varphi , \varepsilon)
\end{equation}
where $G$ is smooth, $2 \pi$-periodic in $\varphi$, satisfying $G=\mathcal{O}(\vert w \vert^2)$ uniformly with respect to the rest of variables when $\vert \varepsilon \vert < \overline \varepsilon$. Moreover the condition of tangency to $\mathcal{F}_0$ and the normalization of the periods implies
\begin{equation}\label{piatta}
G(0,0,\varphi , \varepsilon) \equiv 0 .
\end{equation}
Therefore the flow of $X_{\varepsilon}$ is of the form
\begin{equation}\label{riflusso}
h_t (x,y,\varphi, \varepsilon) = 
\left(
\begin{array}{c}
x \\
y \\
\varphi
\end{array}
\right)
+ t X_0 (x,y,\varphi ) + \underline F (t,x,y,\varphi , \varepsilon)
\end{equation}
where $\underline F$ is smooth, $2 \pi$-periodic with respect to $\varphi$, $\underline F = \mathcal{O}(t^2)$, $\underline F = o(1)$ with respect to $\varepsilon$ and from (\ref{piatta})
\begin{equation}\label{fpiatta}
\underline F (t,0,0,\varphi , \varepsilon) \equiv 0.
\end{equation}
In order to lift $X_{\varepsilon}$ to a vector field $\tilde {X_{\varepsilon}}$ we will lift the flow $h_t :P \rightarrow P$ to a flow $\tilde {h_t} : \tilde {P_{\gamma_0}} \rightarrow \tilde {P_{\gamma_0}}$ according to
\begin{equation}\label{sollevamento}
\sigma \circ \tilde {h_t} = h_t \circ \sigma.
\end{equation}
Lifting arguments of this type are common in algebraic or analytic geometry, where powerful theorem of extension of analytic objects are avaliable: for this reason we leave the (real) analytic case aside and we focus on the case of $C^k$-regularity, $k \geq 2$. In this case questions of the type in (\ref{sollevamento}) are considered in the case of blow up of a point {\it e.g.} in \cite{takens} or in Dumortier's article in \cite{dumortier}. The proof of existence of smooth ${h_t}  :\tilde {P_{\gamma_0}} \rightarrow \tilde {P_{\gamma_0}}$ satisfying (\ref{sollevamento}) could be obtained suitably adapting these results, but a direct proof of it is sufficiently elementary and useful to be presented here. In (\ref{sollevamento}) we freeze $t$-variable and therefore we will prove that
\begin{equation}\label{sollevamento1}
\sigma \circ \tilde {h} = h \circ \sigma.
\end{equation}
As
\[
\sigma \circ \tilde h =(\tilde h_1 , \tilde h_1 \tilde h_2 , \tilde h_3)
\]
(\ref{sollevamento}) reduces to
\begin{equation}\label{zero}
\begin{cases}
\tilde h_1 (x,u,\varphi) &= h_1(x,xu,\varphi) \\
\tilde h_2 (x,u,\varphi) &= \frac{h_2 (x,xu,\varphi)}{h_1 (x,xu,\varphi)} \\
\tilde h_3 (x,u,\varphi) &= h_3 (x,xu,\varphi)
\end{cases}
\end{equation}
so we are left to prove the $C^{k-1}$ extension of the definition in the second equation in (\ref{zero}) when $x \rightarrow 0$. Using that
\[
\begin{cases}
h_1 (x,y,\varphi) &= x + \mathcal{O}(\vert w \vert^2 , \varphi) \\
h_2 (x,y,\varphi) &= y + \mathcal{O}(\vert w \vert^2 , \varphi) \\
h_3 (x,y,\varphi) &= \varphi + \mathcal{O}(\vert w \vert^2 , \varphi)
\end{cases}
\]
and Taylor's formula with Peano's form of the remainder term, we get that the second equation in (\ref{zero}) becomes
\begin{equation}\label{peano}
\tilde h_2 (x,u,\varphi) = \frac{u + x g_2(u, \varphi) +  \cdots +x^{k-1} g_k (u,\varphi) + \frac{G_k (x,xu,\varphi)}{x}}{1 + x f_2 (u,\varphi) + \cdots + x^{k-1} f_k (u,\varphi) + \frac{F_k (x,xu,\varphi)}{x}}
\end{equation} 
where $g_j , f_j$ are polynomials in $u$ of degree at most $j$ and coefficients smoothly depending on $\varphi$, and $G_k , F_k$ are $C^k$-functions, $k \geq 2$ such that $G_k , F_k = \mathcal{O}(\vert w \vert^k)$ uniformly with respect to $\varphi$. The fact that the function in (\ref{peano}) extends as a $C^{k-1}$-function $(x,u,\varphi) \rightarrow u + x (\cdots)$ when $x=0$, {\it i.e.} along the divisor, follows from the following simple
\begin{lemma}
If $G(x,y,\varphi) = \mathcal{O}(\vert w \vert^k)$, and $G$ is  of class $C^k$, then the function $\frac{G(x,xu,\varphi)}{x}$ extends as a function of class $C^{k-1}$ to $\{ x=0 \}$, and is $k-1$-flat at $\{ x=0 \}$.
\end{lemma}
\begin{proof}
Continuos extension of $H(x,u,\varphi)=\frac{G(x,xu,\varphi)}{x}$ to $\{ x=0 \}$ as an identically $0$ function is obvious. The worst (in the sense of the behaviour of the limit as $x \rightarrow 0$) $j$-th derivatives of $H$ is $\frac{\partial^j H}{\partial x^j}$, and we focus on them. When $j=1$
\begin{equation}\label{base}
\frac{\partial H}{\partial x} = \frac{G_x + G_y u}{x} - \frac{G}{x^2}.
\end{equation}
The term $\frac{G}{x^2} = o(1)$ for $k \geq 2$, moreover from uniqueness of the Taylor formula with Peano's remainder, $G_x , G_y = \mathcal{O}(\vert w \vert^{k-1})$ are $C^{k-1}$ functions, hence also  $\frac{G_x + G_y u}{x} = o(1)$ hence proving that $\frac{\partial H}{\partial x}$ smoothly extends to $0$ as $x \rightarrow 0$. In general $\frac{\partial^j H}{\partial x^j}$ decomposes in a sum of ratios having at numerator a function containig derivatives of $G$ of order $l=0, \ldots , k-1$ which is, by the hypothesis on the asymptotic behaviour of $G$ and Taylor's  formula for its derivatives, of type $o(x^{k-l})$, and a denominator which is $x^{k-l}$, hence each of these ratios is $o(1)$.
\end{proof}
We come back now to the proof of Proposition (\ref{risolvi}): we have just proved that $\tilde h_t (x,u,\varphi)$ has a smooth extension up to the divisor $\mathcal{E}$ having, in the considered chart, equation $\{ x = 0 \}$. Moreover is easy to check that, as a consequence of (\ref{fpiatta}), this extension on $\mathcal{E}$ coincides with the restriction to the divisor of the flow of $\tilde X_{\varepsilon}$: therefore $\tilde h_t (x,u,\varphi)$ is a $C^{k-1}$-regular $1$-parameter group of diffeomorphisms of $\tilde {P_{\gamma_0}}$ and therefore defines a $C^{k-1}$ vector field $\tilde {X_{\varepsilon}}$, which coincides on the divisor with $\tilde {X_0}$.

To end the proof we only have to show that $\hat \Sigma = \sigma^{-1}( \Sigma)$, $\Sigma = s(U_1)$, has a smooth exstension up to the divisor and it is transverse to any $\tilde {\mathcal{F}_{\varepsilon}}$, $\vert \varepsilon \vert$ sufficiently small.

Firstly, we characterize $\Sigma$ as an integral surface of an integrable connection $\eta$, singular on $\gamma_0$. Following \cite{chern}, a connection $1$-form in the restriction of $\xi_0$ to the total space $P- \{ \gamma_0 \}$ is given by
$$
\eta = - d \varphi_j + \pi^* A_{j}
$$
where $z_{j}=e^{i \varphi_j}$ is a local coordinate along the fiber and $A_j$ is a gauge potential in $U_j$, $j=0,1$. To represent the section $\Sigma$ in $U_1$ we can take the equation
\begin{equation}\label{equazione}
\eta = 0
\end{equation}
where
\begin{equation}\label{connessione}
\eta =-d \varphi_1
\end{equation}
is the connection corresponding to the gauge potential in $U_1$
$$
A_1 \equiv 0.
$$

The expression of the connection $1$-form (\ref{connessione}) in $U_0$ is then obtained from the transformation law of gauge potentials, {\it cfr.} \cite{chern}:

$$
\pi^* A_0 = i \, d \, log  \,g_{01} = E \, d \theta
$$
where writing $w= x +iy$
we have that $\theta = \arg w = \arctan \frac{y}{x} + \, constant$ or $\theta = \arctan \frac{x}{y} + constant$ and therefore
$$
d \theta = -\frac{y}{x^2 + y^2} dx + \frac{x}{x^2 + y^2} dy.
$$
To get in a neighbourhood of the divisor $\mathcal{E}$ an equation defining $\hat \Sigma$, hence showing that $\hat \Sigma$ smoothly extend up to the divisor, we must lift the equation $\eta =0$ to the blown up space, {\it i.e.} we must write it in coordinates $(x,u,\theta)$ where
\[
\begin{cases}
x & = x \\
y & = ux \\
\theta & = \varepsilon
\end{cases}
\]
and therefore
\[
\pi^* d \theta = -\frac{ux}{x^2 (1+u^2)} dx + \frac{x}{x^2 (1+u^2)} (u dx + x du) = \frac{du}{1+u^2}=d \arctan u
\]
hence equation $\eta =0$ lifts to the smooth $1$-form on $\tilde {P_{\gamma_0}}$
\begin{equation}\label{forma}
\tilde \eta = \sigma^* (d \varphi + E \pi^* d \theta) = d \varphi + E d \arctan u =0
\end{equation}
whose integral manifolds
$$
- \varphi =E\arctan u + \, constant
$$
smoothly extend up to the divisor.

Moreover this smooth manifold $\hat \Sigma$ is transverse to $\tilde X_0$: in fact, $p \rightarrow \tilde \eta \tilde{X_0} (p)$ is smooth on $\tilde {P_{\gamma_0}}$ and
\[
\tilde \eta \tilde X_0  = \eta  X_0   \equiv -1 
\]
on $\tilde {P_{\gamma_0}} - \mathcal{E}$, therefore $\tilde \eta \tilde X_0 \equiv -1$ on $\mathcal{E}$, too.

To end the proof we must prove that, for $\vert \varepsilon \vert$ sufficiently small, $\hat \Sigma$ is transverse to any $\tilde {X_{\varepsilon}}$. This is proved recalling that all $\tilde {X_{\varepsilon}}$'s have the same extension (\ref{esteso}) up to the divisor $\mathcal{E}$ of $X_0$, hence

\[
\tilde \eta \tilde {X_{\varepsilon}} \equiv -1
\]
on $\mathcal{E}$.

Then, from openness of transversality condition, for a fixed open neigbourhood $U$ of $\mathcal{E}$ in $\tilde {P_{\gamma_0}}$ there exists $\overline \varepsilon >0$ such that for $\vert \varepsilon \vert < \overline \varepsilon $ every $\tilde {X_{\varepsilon}}$ is transverse to $\hat \Sigma \cap U$. Up to reducing $\overline \varepsilon$ we can get that $\tilde {X_{\varepsilon}}$ is transverse to the compact manifold $\hat \Sigma \cap (\tilde {P_{\gamma_0}} -U)$, too, and this ends the proof.
\end{proof}
We can now conclude the proof of Theorem (\ref{parallele})
\begin{proof} (of Theorem (\ref{parallele}))
From Proposition (\ref{risolvi}), statement $(ii)$, $\hat \Sigma$ is transverse to $\tilde {X_{\varepsilon}}$ for any $\varepsilon$, $\vert \varepsilon \vert < \overline \varepsilon$. Therefore the Poincar\'e maps relative to the $\tilde {X_{\varepsilon}}$'s
$$
\mathcal{P}_{\varepsilon} : \hat \Sigma \rightarrow \hat \Sigma
$$
are pointwise periodic, and the map
\[
\begin{cases}
\mathcal{P} & : \hat \Sigma \times ]-\overline \varepsilon , \overline \varepsilon[ \rightarrow \hat \Sigma \times ]-\overline \varepsilon , \overline \varepsilon[ \\
\mathcal{P}(p,\varepsilon) & =(\mathcal{P}_{\varepsilon}(p) , \varepsilon)
\end{cases}
\]
is pointwise periodic, too. From a theorem by D. Montgomery \cite{montgomery} $\S V$ there exists $N \in \mathbb{N}$ such that
\begin{equation}\label{periodica}
\begin{cases}
\mathcal{P}^{(k)} & \neq \Id \, k=1, \dots , N-1 \\
\mathcal{P}^{(N)} = \Id 
\end{cases}
\end{equation}
and it is easy to see that $N=2E$.

Let $\phi^t_{\varepsilon}(\cdot)$ be the flow of $\tilde {X_{\varepsilon}}$ and let

\[
T_{\varepsilon} : P \rightarrow \mathbb{R}^+
\]
be defined as
$$
T_{\varepsilon} (p) =\tau (\varepsilon , p ) + \sum_{i=1}^{2E-1} \tau_i (\varepsilon , p)
$$
where $\tau_i (\varepsilon , p) \in ]0, 2 \pi ]$ is the positive number such that $\phi_{\varepsilon}^{t}(\phi_{\varepsilon}^{\tau_{i-1} (\varepsilon , p)}) \notin \hat \Sigma$, $\phi_{\varepsilon}^{\tau_{i}(\varepsilon , p)}(\phi_{\varepsilon}^{\tau_{i-1} (\varepsilon , p)}) \in \hat \Sigma$, $i=1, \ldots ,2E-1$, $\tau_0 (\varepsilon , p) =0$ and $\phi_{\varepsilon}^{\tau (\varepsilon , p )} (\phi_{\varepsilon}^{\sum_{i=1}^{2E-1} \tau_i (\varepsilon , p)}(p)) = p$. In other words, $T_{\varepsilon} (p)$ is the (minimal )period function of the flow of $\tilde {X_{\varepsilon}}$ for points $p \notin \mathcal{E}$, while is $2E$-times the minimal period for points $p \in \mathcal{E}$: the above analytic definition and the Implicit Function Theorem implies that $T_{\varepsilon}(\cdot)$ is smooth so
$$
\hat X_{\varepsilon} = \frac{2 \pi}{T_{\varepsilon}} \tilde {X_{\varepsilon}}
$$
is smooth, too. We still name $\tilde {X_{\varepsilon}}$ these reparametrized vector fields, and note that they are isochronous of common period $2 \pi$. The flows of the $\tilde {X}_{\varepsilon}$'s blow down to the flows of the $X_{\varepsilon}$'s, and therefore, still denoting $\phi^t_{\varepsilon}(p)$ the flow of $X_{\varepsilon}$ with initial datum $p \in P$, we can define the smooth functions
\begin{equation}\label{diffeo}
\varphi_{\varepsilon} (p)=\phi_0^{\tau (\varepsilon , p)}(p(\varepsilon)).
\end{equation}
Each of these map is a smooth diffeomorphism, extending smoothly to the identity on $\gamma_0$, which maps each integral curve of $X_{\varepsilon}$ through $p \in \Sigma$ to the integral curve of $X_0$ through $p$. To prove that $\varphi_{\varepsilon}$ is a smooth isomorphism of circle bundle we define the $S^1$-actions
\begin{equation}\label{azione}
\begin{cases}
*_{\varepsilon} & : S^1 \times P \rightarrow P \\
\theta *_{\varepsilon} p & = \phi_{\varepsilon}^{\theta}
\end{cases}
\end{equation}
$\theta \in \frac{\mathbb{R}}{2 \pi \mathbb{Z}} \simeq S^1$.
The group property of the flow of an autonomous differential equation implies
$$
\tau (\varepsilon , \phi_{\varepsilon}^{\theta}(p)) = \tau (\varepsilon , p) + \theta
$$
therefore
$$
\varphi_{\varepsilon} (\theta *_{\varepsilon}p)=\varphi_{\varepsilon}(\phi_{\varepsilon}^{\theta}(p))=\phi_0^{\tau (\varepsilon , p) + \theta}(p(\varepsilon)) = \phi_0^{\theta} \circ \varphi_{\varepsilon} (p)= \theta *_0 \varphi_{\varepsilon} (p)
$$
which proves that (\ref{azione}) defines a circle bundle $\xi_{\varepsilon}$ whose fibers are the leaves of $\mathcal{F}_{\varepsilon}$ and $\varphi_{\varepsilon} : \xi_{\varepsilon} \rightarrow \xi_0$ is a bundle isomorphism.

\end{proof}
We can prove now Corollary (\ref{analitico})
\begin{proof}

Let $\mathcal{U}= \{ U_{\alpha}'\}$ an open cover of $M$ and let
$$
\varphi_{\alpha} : \pi^{-1}(U_{\alpha}' \times S^1) \rightarrow U_{\alpha}' \times S^1
$$
be trivializing diffeomorphisms of the bundle $\xi_0$.

Let $\{ U_{\alpha} \}_{\alpha}$ be a refinement of $\{ U_{\alpha}' \}_{\alpha}$ and let
$$
D_{\alpha} = \varphi^{-1}_{\alpha} (U_{\alpha}\times\{ 1 \})
$$
and
$$
D_{\alpha}' = \varphi^{-1}_{\alpha} (U_{\alpha}' \times\{ 1 \}).
$$
For $p \in P$
$$
<X_0 (p) > + T_p D_{\alpha} = T_p P
$$
and openess of this property and closeness of $P$ imply that
\begin{equation}\label{trasversalita}
<X_{\varepsilon} (p) > + T_p D_{\alpha}' = T_p P
\end{equation}
holds for $\vert \varepsilon \vert < \overline \varepsilon $, for any $p \in P$ and for $\overline \varepsilon$ positive and sufficiently small. Denoting as usual $\phi^t_{\varepsilon} (\cdot)$ the flow of $X_{\varepsilon}$, using the transversality property (\ref{trasversalita}), the Implicit Function Theorem and a careful choice of the coverings $\{ U_{\alpha} \}_{\alpha}$, $\{ U_{\alpha}' \}_{\alpha}$, we get the existence of smooth first return times
$$
t_{\alpha , \varepsilon} : D_{\alpha} \rightarrow \mathbb{R}^+
$$
such that
$\phi^t_{\varepsilon}(p) \notin \D_{\alpha}$ if $0< t < t_{\alpha , \varepsilon} (p) $ and $\phi^{t_{\alpha , \varepsilon}}_{\varepsilon}(p) \in \D_{\alpha}$. Uniqueness of the implicit function implies that $t_{\alpha , \varepsilon} (p) = t_{\beta , \varepsilon}(p)$ if $p \in U_{\alpha} \cap U_{\beta}$. Let
$$
\mathcal{P}_{\alpha , \varepsilon} : \D_{\alpha} \rightarrow D_{\alpha}'
$$
be the Poincar\'e map, where
\begin{equation}\label{poincare}
\mathcal{P}_{\alpha , \varepsilon} (p)=\phi^{t_{\alpha , \varepsilon}(p)}{\varepsilon}(p).
\end{equation}
Of course
$$
\mathcal{P}_{\alpha , 0} = identity \, on \, D_{\alpha}'.
$$
Let
$$
Fix \, \mathcal{P}_{\alpha , \varepsilon } = \, set \, of \, fixed \, points \, of \, \mathcal{P}_{\alpha , \varepsilon }
$$
and define for any $\varepsilon$, $\vert \varepsilon \vert < \overline \varepsilon$ the set
$$
\mathcal{Z}_{\varepsilon} = \cup_{\alpha} \{ \phi^t_{\varepsilon} (p): p\in Fix \, \mathcal{P}_{\alpha , \varepsilon }, 0\leq t \leq t_{\alpha , \varepsilon} \}.
$$
If $P$, $\varepsilon \rightarrow X_{\varepsilon}$ are real analytic, $\mathcal{Z}_{\varepsilon}$ is real analytic, too, and
$$
\mathcal{Z}=\cup_{\vert \varepsilon \vert < \overline \varepsilon} \mathcal{Z}_{\varepsilon}
$$
is real analytic as well. Of course, $\mathcal{Z}_0 = P$.

From Theorem (\ref{seifert}) $\mathcal{Z}_{\varepsilon} \neq \emptyset$ for every $\varepsilon$. From a theorem of Bruhat and H. Cartan \cite{bc} there exists a real analytic curve $\varepsilon \rightarrow \gamma_{\varepsilon}$, {\it i.e.} a real analytic curve of Seifert's leaves: this remark ends the proof.
\end{proof}
\begin{remark}
We whish to end this section with an observation concerning Theorem (\ref{parallele}), whose statement implies conjugacy of two foliations when they are sufficiently $C^1$-close, one of them has leaves which are the fibers of a circle bundle, and they are tangent along a fiber. If the perturbation $\varepsilon \rightarrow \mathcal{F}_{\varepsilon}$ of the foliation $\mathcal{F}_0$ whose leaves are the fiber of the circle bundle $\xi_0$ was {\it a priori} known to be made by foliations by circles generated by circle bundles, then Theorem (\ref{parallele}) would be rather obvious and consequence of the fact that circle bundles are classified by one integer-valued invariant, the Euler number, {\it once a open covering of the base space has been fixed}, a condition that by Lemma (\ref{localizzazione}) is equivalent to tangency at $\gamma_0$ of the $\mathcal{F}_{\varepsilon}$. It is perhaps less obvious that the rigidity property in Theorem (\ref{parallele}) still holds true without imposing {\it a priori} that the foliations $\mathcal{F}_{\varepsilon}$ arise from circle bundles.

\end{remark}

\section{Final remarks and an example (by Thurston) in dimension $4$}
This final section is devoted to some remarks on the relevant hypotheses we made in Theorem (\ref{parallele}) and which are listed below (we do not consider here orientability assumptions):

\begin{itemize}\label{lista}
\item [(i)] existence of a smooth curve of Seifert's leaves for the perturbation $\varepsilon \rightarrow \mathcal{F}_{\varepsilon}$ \\
\item [(ii)] the fact that the unperturbed foliation $\mathcal{F}_0$ is defined by the fibers of a circle bundle \\
\item [(iii)] compactness of the fibers of the unperturbed bundle \\
\item [(iv)] the fact that $dim \, P =3$ (dimensionality hypothesis).
\end{itemize}

A part from hypothesis $(iv)$, which is considered in some details, we limit ourselves to brief comments concerning $(i), (ii), (iii)$.

Hypotheses $(ii)$, $(iii)$ can be weakened, {\it e.g.} one can suppose that the unperturbed foliation is a Seifert fibration of Seifert manifolds, and still a result similar to Theorem (\ref{parallele}) holds: we hope to deal with this subject in a nearly future.

Hypotheses $(i)$, $(iv)$ are related through Corollary (\ref{analitico}). Some generalizations obtained weakening $(i)$ are still possible, and left for future work, too, but basically the main question concerning $(i)$ is the existence of an example of a $1$-parameter family $\varepsilon \rightarrow \mathcal{F}_{\varepsilon}$ of foliations by circle satifying the Seifert Stability Theorem which does not admit a smooth curve of Seifert's leaves: if such an example exists, its dynamics should be very interesting. Corollary (\ref{analitico}) shows that such example cannot exist in the real analytic setting and in dimension $3$.

The dimensionality hypothesis $(iv)$ is necessary for Theorem (\ref{parallele}) to hold. This is proved by a celebrated example by W. Thurston, originally presented by D. Sullivan in \cite{sullivan} and later described by Godbillon \cite{godbillon} and D.B.A. Epstein in Appendix $1$ in \cite{besse}. Such example shows that there exists a $1$-parameter family $\varepsilon \rightarrow \mathcal{F}_{\varepsilon}$ of foliations by circles of a $4$-manifold $P$, deforming a circle bundle $\xi_0$ with total space $P$, which do not satisfy the rigidity property proved in Theorem (\ref{main}). We slightly modify such example in order to prove that question $(I)$ posed at the end of the introduction of this article, which has negative answer in the case of $1$-parameter families of oscillators $(P, \mathcal{F}_{\varepsilon})$ when $dim P=3$, has instead positive answer when $dim P=4$: we dare to add our version of Thurston's example to those quoted above just to explain this point. We briefly recall $(I)$: it asks for the existence of a smooth family of oscillators $\varepsilon \rightarrow (P , \mathcal{F}_{\varepsilon})$, $0 \leq \varepsilon \leq 1$ such that the leaves of $\mathcal{F}_0$ are the fiber of a circle bundle $\xi_0$ and the leaves of $\mathcal{F}_1$ are the fiber of a circle bundle $\xi_1$, and these two bundles are not isomorphic, in fact they have not homotopic base spaces.

The building blocks of Thurston's example are two circle bundles over $T^2$
\[
\begin{cases}
\eta : & S^1 \hookrightarrow S(T^2) \rightarrow^{\pi_{\eta}} T^2 \\
\mu : & S^1 \hookrightarrow H \rightarrow^{\pi_{\mu}} T^2
\end{cases}
\] 
where $\eta$ is the (trivial) unit tangent bundle of the flat torus, while
\begin{equation}\label{heisenberg}
H=\frac{H_3(\mathbb{R})}{H_3(\mathbb{Z})}
\end{equation}
where $H_3(\mathbb{R})$, $H_3(\mathbb{Z})$ are respectively the $3$-dimensional Heisenberg group over the real and integer numbers. We denote $X_{\eta}$, respectively $X_{\mu}$, the isochronous infinitesimal generators of $\eta , \mu$.

These two bundles are coupled togheter through the {\it fiber product} principal bundle
$$
\eta \times_{T^2} \mu : T^2 \hookrightarrow P= S(T^2) \times_{T^2} H \rightarrow^{\pi} T^2
$$
where
$$
P= \{ (p,p') \in S(T^2) \times H : \pi_{\eta}(p)=\pi_{\mu}(p')=\pi (p,p') \}.
$$
This embedding of $P$ in $S(T^2) \times H $ implies that
$$
T_{(p,p')} P < T_{(p,p')} (S(T^2) \times H) \simeq T_p S(T^2) \times T_{p'} H
$$
therefore a vector field
$$
X: P \rightarrow TP
$$
can be written, with slight abuse of notation, as
\begin{equation}\label{splitting}
X(p,p') = (S(p,p'),T(p,p'))
\end{equation}
where $\pi_{\eta}(p)=\pi_{\mu}(p')$ and $S(p,p') \in T_p S(T^2)$, $T(p,p') \in T_{p'} H$. The description of our modification of Thurston's example is completed by the introduction of two {\it marginal bundles} of $\eta \times_{T^2} \mu$ defined by
\[
\begin{cases}
\xi_1 & : S^1 \hookrightarrow P \rightarrow^{\pi_1} H \\
\xi_2 & : S^1 \hookrightarrow P \rightarrow^{\pi_1} S(T^2)
\end{cases}
\]
generated by the $S^1$-actions
\[
\begin{cases}
(g_1 ,1)*_{\xi_1} (p,p') &= (g_1 *_{\eta} p ,p') \\
(1 , g_2) *_{\xi_2} (p,p')&=(p,g_2 *_{\mu} p')
\end{cases}
\]
$g_1 , g_2 \in S^1$. Adopting notation (\ref{splitting}) the isochronous infinitesimal generators $Y_1 , Y_2$ of $\xi_1 , \xi_2$ are $Y_1 = (X_{\eta},0)$, $Y_2 = (0 , X_{\mu})$.

Our goal will be to define a real analytic family of foliation by circles of $P$, parametrized by $\lambda$, defined by the integral curves of the vector fields
$$
X_{\lambda} = (\alpha_1 (\lambda) S_{\lambda} ,\alpha_2 (\lambda) T_{\lambda}) : P \rightarrow TP
$$
$\lambda \in [0,\infty]$, $\alpha_1 (\lambda), \alpha_2 (\lambda) >0$, such that $X_0=Y_1$,  $X_{\infty}=Y_2$, hence showing that the hypothesis on the dimension of the total space $P$ is esssential in Theorem (\ref{main}), and moreover question $(I)$ at the end of the introduction has answer in the affirmative.

Let $\tilde S_{\lambda} : \mathbb{C} \times S^1 \rightarrow T(\mathbb{C} \times S^1)$
\[
\tilde S_{\lambda} :
\begin{cases}
\dot z &= \zeta \\
\dot \zeta &= \frac{i}{\lambda} \zeta
\end{cases}
\]
then $\vert \zeta (t) \vert \equiv \vert \zeta(0) \vert$. Therefore defining
$$
S_{\lambda}(z,\zeta) = \tilde S_{\lambda}(z,\zeta)
$$
we get a real analytic vector field
$$
S_{\lambda}(z,\zeta) : S(T^2) \rightarrow T S(T^2)
$$
where $T^2 = \frac{\mathbb{C}}{2 \pi \mathbb{Z} \times 2 \pi \mathbb{Z}}$. For any $\lambda \in ]0, \infty[$  the integral curves of $S_{\lambda}$ define a foliation by circles $\gamma = \gamma (z_0, \zeta_0, \lambda)$, $z_0$ center, $\lambda$ radius, $\zeta_0$ initial velocity, determining a circle of the foliation. When $\lambda = \infty$ the leaves of the foliation of $S(T^2)$ are circles if and only if $\frac{\Im(\zeta_0)}{\Re(\zeta_0)} \in \mathbb{Q}$. The $\pi_{\eta}$-projection of the curves $\gamma$ are the {\it drift curves} of the sought dynamics on $P$. The {\it synchronization argument} by Thurston, provides a way to lift the drift curves to closed curves $\Gamma \subset P$ defining for any $\lambda \in ]0, \infty]$ a foliation by circles $\mathcal{F}_{\lambda}$  of $P$. We slightly modify this construction to get a $1$-parameter family of foliations $\mathcal{F}_{\lambda}$ which coincides for $\lambda =0$ with the fibers of $\xi_1$, respectively for $\lambda =\infty$ with the fibers of $\xi_2$.

Let $U \times S^1$ be a trivializing solid torus for $\mu$, let $x_U = x_U \, mod \, 2 \pi$ be a local coordinate along $S^1$. Adopting the normalization condition $\varphi (X_{\mu})\equiv -1$, {\it cfr.} \cite{chern}, a connection $1$-form on $\mu$ in local coordinates is
$$
\varphi = - d x_U + {\pi_{\mu}}^* \theta_U. 
$$
We will define $\Gamma$ such that
\begin{equation}\label{solleva}
\pi_{\mu} : \Gamma \rightarrow \gamma
\end{equation}
and it is a closed integral curve of the partially (un)coupled vector field $X_{\lambda} = (\alpha_1 (\lambda) S_{\lambda},\alpha_2 (\lambda) T_{\lambda})$ where $S_{\lambda}$ is the previously defined true vector field on $S(T^2)$ and $\alpha_1 , \alpha_2$ are suitably defined positive functions. The definition of $T_{\lambda} : P \rightarrow TH$ goes as follows. We recall, {\it cfr.} \cite{chern}, that the {\it phase} of the lift (\ref{solleva}) is
$$
\Delta x (\gamma) = \int_{int \, \gamma} \Theta - \int_{\Gamma} \varphi
$$
where $\Theta$ is the curvature $2$-form of the connection $\varphi$ and $int \, \gamma$ is the interior of $\gamma \subset T^2$. It is customary to name
$$
\int_{int \, \gamma} \Theta = geometric \, \, phase
$$
$$
\int_{\Gamma} \varphi = dynamical \,  \,phase .
$$
The condition that $\Gamma$ is a closed curve becomes $\Delta x (\gamma) = 2 \pi k$, $k \in \mathbb{Z}$:  we choose $k=0$ as Thurston did \cite{sullivan}, getting $\Delta x (\gamma) =0$ hence
$$
\int_{int \, \gamma} \Theta = \int_{\Gamma} \varphi.
$$
For the connection $1$-form $\varphi$ we again follow Thurston's choice \cite{sullivan}, defining in $U=\frac{\mathbb{C}}{2 \pi \mathbb{Z} \times 2 \pi \mathbb{Z}}$
$$
\varphi = -d x_U + {\pi_{\mu}}^* (x dy).
$$
Hence
$$
\int_{int \, \gamma} \Theta = \pi \lambda^2
$$
and equality of geometric and dynamical phases implies
\begin{equation}\label{fasi}
\int_{\Gamma} \varphi = \pi \lambda^2 .
\end{equation}
Putting
\begin{equation}\label{thurston}
X_{\lambda} = (\alpha_1 (\lambda) S_{\lambda} , \alpha_2 (\lambda) X_{\mu})
\end{equation}
where $\alpha_1 , \alpha_2 : [0,\infty] \rightarrow \mathbb{R}^+$ are smooth functions, from (\ref{fasi})
$$
\pi \lambda^2 = \int_0^{\frac{2 \pi \lambda}{\alpha_1 (\lambda)}} \alpha_2 (\lambda) \, dt = 2 \pi \lambda \frac{\alpha_2}{\alpha_1}
$$
we get
\begin{equation}\label{blocco}
\frac{\alpha_2 (\lambda)}{\alpha_1 (\lambda)} = \frac{\lambda}{2} .
\end{equation}
For any $\lambda \in ]0,\infty[$, (\ref{thurston}) satisfying (\ref{blocco}) defines a foliation by circles. Thurston's original choice: $\alpha_2 (\lambda) = \frac{\lambda}{2}$, $\alpha_1 (\lambda) \equiv 1$ gives for $\lambda = \infty$ the foliation by circles defined by the fibers of $\xi_2$, if $X_{\lambda}$ is reparametrized by multiplication for $\frac{2}{\lambda}$. We modify this choice, still keeping (\ref{blocco}), defining
\[
\begin{cases}
\alpha_2 (\lambda) &= \frac{\lambda}{2 + \lambda} \\
\alpha_1 (\lambda) &= \frac{2}{2 + \lambda}
\end{cases}
\]
therefore letting $\lambda \rightarrow 0^+$ we get in (\ref{thurston})
$$
X_0 = (X_{\eta} , 0) =Y_1
$$
while letting $\lambda \rightarrow \infty$ we obtain
$$
X_{\infty} = (0 , X_{\mu})=Y_2
$$
hence proving that $\lambda \rightarrow X_{\lambda}$, $\lambda \in [0,\infty]$ define a $1$-parameter family of foliations by circles of $P$ interpolating those two generated by the fibers of the marginal bundles $\xi_1 , \xi_2$. These two bundles cannot be isomorphic, and their base spaces are not homotopic, hence giving the sought example showing that if the hypothesis that $dim P=3$ is dropped Theorem (\ref{main}) is false and moreover providing an example of a smooth $1$-parameter family of oscillators on a a $4$-dimensional manifolds connecting two oscillators with not homeomorphic base spaces.

In fact, is sufficient to observe that $H_1 (S(T^2) , \mathbb{Z})) = \mathbb{Z}^3$, while from (\ref{heisenberg})
\[
H_1(H,\mathbb{Z}) \simeq H_3 (\mathbb{Z})^{abel} = \mathbb{Z}^2
\] 
where $H_3 (\mathbb{Z})^{abel}$ is the abelianization of $H_3 (\mathbb{Z})$ and its computation follows easily from generators and relations of this group.

If needed, this example can be projected, as in \cite{sullivan}, to a real analytic vector field
\[
\begin{cases}
X &:P \times S^1 \rightarrow T(P \times S^1) \\
X((p,p'), \theta) &= X_{\lambda (\theta)} (p,p')
\end{cases}
\]
where for instance
$$
\lambda (\theta) = \cot \theta .
$$
Each
$$
X_{\lambda (\theta)} : P \times \{ \theta \} \rightarrow TP
$$
foliates $P$ by circles, and for $\theta = \frac{\pi}{2}$, respectively for $\theta = \pi$, the foliation is generated by the fibers of $\xi_1$, respectively $\xi_2$. This is the sought example showing that the dimensionality assumption in Theorem (\ref{main}) is necessary.

\end{document}